\documentclass[12pt]{article}

\usepackage{amsmath,amssymb,mathrsfs,graphicx,bbm,amsfonts,amsthm,mathtools}
\usepackage{wrapfig}
\usepackage[utf8x]{inputenc}
\usepackage{dsfont}
\usepackage[colorlinks=true, linkcolor=blue, citecolor=blue, urlcolor=blue]{hyperref}

\usepackage{color}
\usepackage{comment}
\usepackage{enumitem}
\usepackage{dsfont}
\usepackage[normalem]{ulem}
\usepackage[english]{babel}
\usepackage{soul}
\usepackage{xcolor}

\usepackage{todonotes}

\usepackage{comment}

\usepackage[scriptsize,bf]{caption}

\newtheorem{theorem}{Theorem}[section]
\newtheorem{lemma}[theorem]{Lemma}
\newtheorem{proposition}[theorem]{Proposition}
\newtheorem{corollary}[theorem]{Corollary}
\newtheorem{definition}[theorem]{Definition}

\newtheorem{remark}[theorem]{Remark}

\newtheorem{assumption}[theorem]{Assumption}
\newtheorem{claim}[theorem]{Claim}

\DeclareMathAlphabet{\mathbfit}{OML}{cmm}{b}{it}

\makeatletter
\@addtoreset{equation}{section}
\makeatother

\newcommand{\ER}{Erd\H{o}s-Rényi\ }

\newcommand{\cadlag}{\ifmmode\mathcal D\else c\`adl\`ag \fi}

 \newcommand{\ba}{\begin{array}}
 \newcommand{\ea}{\end{array}}
 \newcommand{\bea}{\begin{eqnarray}}
 \newcommand{\eea}{\end{eqnarray}}
 \newcommand{\be}{\begin{equation}}
  \newcommand{\ee}{\end{equation}}

 \def \R {{\mathbb R}}
 
 \def \N {{\mathbb N}}
 \def \P {{\mathbb P}}

 \def \a {{\alpha}}

 \def \z {{\z}}

 \def \z {{\zeta}}

 \def \1{\mathbbm{1}} 

\begin{document}
\title{ERGMs on block models}
\author{E. Magnanini\footnote{Weierstrass Institute for Applied Analysis and Stochastics, Anton-Wilhelm-Amo-Straße 39,, 10117 Berlin, Germany.}}
\maketitle
\abstract{We extend the classical edge-triangle Exponential Random Graph Model (ERGM) to an inhomogeneous setting in which vertices carry types determined by an underlying partition. This leads to a block-structured ERGM where interaction parameters depend on vertex types.
We establish a large deviation principle for the associated sequence of measures and derive the corresponding variational formula for the limiting free energy. In the ferromagnetic regime, where the parameters governing triangle densities are nonnegative, we reduce the variational problem to a scalar optimization problem, thereby identifying the natural block counterpart of the replica symmetric regime.
Under additional restrictions on the parameters, comparable to the classical Dobrushin uniqueness region, we prove uniqueness of the maximizer and derive a law of large numbers for the edge density.}
\tableofcontents
\section{Introduction} \label{intro}
Social networks, like many biological and technological networks, are known to exhibit a high degree of \emph{clustering}, also referred to as \emph{transitivity}.
Informally, if two nodes share a common neighbor, they are more likely to be linked to each other as well. 
A central modelling challenge is to relate such \emph{local features}, captured by the frequency of small subgraphs and therefore easily measurable, to the \emph{global structure} of the network.
Among probabilistic models for social networks (see e.g.\cite{robins2014exponential}), \emph{Exponential Random Graph Models} (ERGMs) have a particularly transparent statistical justification. 
More precisely, if one prescribes the empirical values of finitely many network statistics (such as the number of edges, triangles, or other small subgraphs), then among all probability distributions on graphs consistent with these constraints, the entropy-maximizing distribution is a Gibbs measure \cite{jaynes1957information}.
Focusing on unconstrained, undirected ERGMs, a substantial body of rigorous probabilistic results has been developed. These include the derivation of the limiting free energy and the analysis of its phase diagram, as well as more general results on the asymptotic structure and typical behavior of such graphs (see, e.g., \cite{CD, CDey, RY, MY, AZ, eldan2018exponential}).
A complementary line of research investigates fluctuation phenomena and limit theorems for subgraph densities, obtained both through statistical mechanics techniques (see, e.g., \cite{BCM, MP, BCM22, magnanini2025standard, MX}) and via Stein’s method, the latter also providing quantitative normal approximations \cite{fang2024normal, SaSi, winstein2025quantitative}.

Classical ERGMs are typically \emph{homogeneous}, in the sense that their law is invariant under relabellings of the vertex set. Real networks, however, often exhibit structural heterogeneity: vertices carry attributes (or types) that significantly influence link formation.
In this work we study \emph{inhomogeneous ERGMs} in which heterogeneity is introduced through an underlying partition of the vertex set into finitely many blocks. Each vertex is assigned a type (or color) according to the block to which it belongs. Thus, in the underlying reference measure, the probability to create an edge between two vertices depends on their types. Models of this kind were originally introduced in the physics literature \cite{fronczak2013exponential} to describe community organization, where types may represent political orientation, scientific field, thematic similarity of webpages, or other attributes.

From a probabilistic point of view, inhomogeneous ERGMs can be seen as exponential tilts of dense inhomogeneous \ER graphs, in the same way that classical ERGMs arise as tilts of the homogeneous \ER model. 
Vertex types enter the \emph{Hamiltonian} through the interaction parameters that weight the various subgraph densities, depending on the types of the vertices involved in those subgraphs.
For instance, one may assign a higher weight to triangles formed by vertices of the same type than to mixed-type configurations, thereby reflecting an underlying notion of proximity, such as shared interests or geographical closeness.

Among recent developments on block-structured graph models, we mention the derivation of a Large Deviation Principle (LPD) in \cite{GP, borgs2020large}, together with advances in the closely related theory of \emph{probability graphons} \cite{abraham2025probability, dionigi2025large}, which extend the classical graphon framework to settings where edges carry additional decorations.

\subsection{Our contribution}
The main novelty of this paper is to provide a step further towards the understanding of inhomogeneous ERGMs. 
We show that the variational principle available for the free energy for homogeneous ERGMs can be transferred to this richer setting, yielding a well-posed variational problem under minimal structural assumptions on the underlying partition. This leads to an explicit characterization of the maximizers and to a natural block-structured counterpart of the \emph{replica symmetric regime}. Outside this regime, we expect effects associated with \emph{symmetry breaking} and phase transitions, whose analysis we defer to future work.
The techniques of the proofs rely on classical tools, suitably adapted to the block structure of the model. 
In what follows, we briefly list our main results. 
\begin{enumerate}
    \item In Thm.~\ref{thm:LDP-free-energy}, we prove a large deviation principle for the associated sequence of measures. From the corresponding rate function, we derive the variational formula for the free energy; an alternative direct derivation is provided in Appendix~\ref{alternative_der_fe}.
    
    \item In Theorem~\ref{thm:EL-alphaijk}, we derive the Euler-Lagrange equations that characterize any optimizer of the variational formula, without imposing restrictions on the parameter regime. 
    
    \item In Thm.~\ref{thm:scalar_probl}, assuming nonnegativity of parameters tuning the triangle densities, we reduce the variational problem for the free energy to a scalar optimization problem, leading to the fixed-point system \eqref{eq:fixed-point-scalar}. In Thm.~\ref{prop:small-alpha-tensor}, we identify a restricted parameter regime (comparable to the so-called Dobrushin uniqueness region) under which this system admits a unique solution. Combining these results, we obtain our uniqueness statement in Cor.~\ref{cor:unique-maximiser}.
    \item Finally, relying on the uniqueness result, we establish a strong law of large numbers for the edge density, extending \cite[Thm.~3.10]{BCM}.
     
\end{enumerate}

\subsection{Structure}

The rest of the paper is organized as follows.

\begin{enumerate}
    \item In Sec.~\ref{sec_results} we introduce the model, the terminology, and our main assumption on the partition (see Assumption~\ref{ass:partition}).
    
    \item Subsec.~\ref{sec:color_Graphons} contains the main definitions, including the notion of \emph{colored graphons} (and their associated space), as well as Def.~\ref{def:eq_Rel} of the equivalence relation.
    
    \item In Subsec.~\ref{sec:main_res} we state our main results, while Sec.~\ref{sec_proofs} is devoted to their proofs.
    
    \item Finally, Appendix~\ref{sec_append} collects auxiliary lemmas and provides a direct derivation of the free energy, as an alternative to applying the Legendre transform to the rate function.
\end{enumerate}

\section{Preliminaries}\label{sec_results}

\subsection{The model}\label{subs:model}
Let $n\in\mathbb N$. We denote by $\mathcal G_n$ the set of all simple undirected graphs on the vertex set
$[n]:=\{1,\dots,n\}$.
For each $G\in\mathcal G_n$, let
\[
X=(X_{ij})_{1\le i,j\le n}
\]
be its adjacency matrix, i.e.\ the $n\times n$ symmetric matrix with entries in
$\{0,1\}$ defined by
\[
X_{ij}
=
\begin{cases}
1, & \text{if } \{i,j\}\text{ is an edge of }G, \qquad i\neq j\\
0, & \text{otherwise}
\end{cases}.
\]
We denote by $\mathcal A_n$ the set of all such adjacency matrices.
Beyond the underlying graph, we endow the vertex set with a finite coloring.
This is achieved by fixing a partition of the vertex set $[n]$ into $k$
(nonempty) classes, which represent the different colors.
Formally, we consider a collection
\begin{equation}\label{def:part_1}
\mathbf B^{(k)}\equiv \mathbf B^{(k)}_{n}=\bigl(B^{(n)}_1,\dots,B^{(n)}_k\bigr)
\end{equation}
of disjoint subsets of $[n]$ whose union is $[n]$.
The interpretation is that two vertices $i,j\in[n]$ have the same color if and only if
they belong to the same element (\emph{block}) $B^{(n)}_a$, $a\in[k]$ of the partition. In what follows, we consider only partitions $\mathbf B^{(k)}_{n}$ satisfying the following assumption (which may be viewed as the analogue of \cite[Eq.~2.2]{BR}).

\begin{assumption}[Partitions] \label{ass:partition}
Fix $k\in\mathbb N$ and a vector of block lengths
${\bf b} := (b_1,\dots,b_k)$ with $b_i>0$ and $\sum_{i=1}^k b_i = 1$.
Let $\mathcal B^{(k)} = (\mathcal B_1,\dots,\mathcal B_k)$ be a partition of
$[0,1]$ into consecutive intervals with Lebesgue measure $\lambda(\mathcal B_i) = b_i$.\\
For each $n\ge1$ consider a partition
$\mathbf B^{(k)} = (B^{(n)}_1,\dots,B^{(n)}_k)$ of $[n]$
into consecutive intervals with cardinalities $w^{(n)}_i:=|B^{(n)}_i|$ satisfying
\begin{equation}\label{eq:block-proportions}
\sum_{i=1}^k w^{(n)}_i = n \quad \text{and} \quad  \frac{w^{(n)}_i}{n} \longrightarrow b_i
  \qquad\text{as }n\to\infty,\quad i=1,\dots,k.
\end{equation}

Equivalently, $\mathcal B^{(k)}$ may be viewed as the limiting coloring of the unit interval $[0,1]$.
\end{assumption}

We denote by $\mathcal{G}^{(\bf{B})}_{n}:=\{(G,{\bf{B}}^{(k)}): G\in \mathcal{G}_n\}$ the set of colored graphs. 
In view of the bijection between the sets $\mathcal{A}_n$ and $\mathcal{G}_n$, we may regard at the Hamiltonian of the \emph{edge-triangle model on blocks} as a function from
$\mathcal{A}^{(\bf{B})}_{n}:=\{(X,{\bf{B}}^{(k)}): X\in \mathcal{A}_n\}$
 onto $\mathbb{R}$,  defined by
 
\begin{align}\label{Hamilt_ERG2}
\mathcal{H}^{(\mathbf B)}_{n;\boldsymbol\alpha,\mathbf h}(X)
&:= \frac{1}{n}\sum_{i,j,\ell\in[k]} \alpha_{ij\ell}
\sum_{\substack{u\in B_i,\ v\in B_j,\ w\in B_\ell \\ u<v<w}}
X_{uv}\,X_{vw}\,X_{uw}\\
&\hspace{4cm}+\sum_{i,j \in [k]} h_{ij}\sum_{\substack{u\in B_i,\ v\in B_j\\ u<v}} X_{uv},\notag
\end{align}
where $\boldsymbol{\alpha}:=(\alpha_{ij\ell})_{i,j,\ell\in[k]}\in \R^{k\times k\times k}$ is a collection of
triangle interaction parameters that encode the relative strength of triangle
interactions across communities (for instance, triangles formed by vertices in the same
community may be favored over mixed configurations). 
We assume that the triangle parameters 
$\boldsymbol\alpha=(\alpha_{ij\ell})_{i,j,\ell\in[k]}$
are symmetric under permutations of the indices (for instance, triangles with one vertex in $B_1$ and two vertices in $B_3$
are assigned the same weight, i.e.
$\alpha_{133}=\alpha_{313}=\alpha_{331}$).
Since the constraint $u<v<w$ ensures that each triangle (resp. edge) is counted exactly once in 
\eqref{Hamilt_ERG2}, this symmetry is purely a structural modeling assumption 
and does not influence the combinatorial counting.
Similarly, 
$\mathbf h:=(h_{ij})_{i,j\in[k]} \in \R^{k\times k}$ is a symmetric matrix of edge weights.

\begin{remark}[Notation]\label{notation}
To avoid overloading the notation, we will not always make the dependence on the
partition explicit. Strictly speaking, the argument of \eqref{Hamilt_ERG2} should be the pair
$(X,\mathbf B^{(k)})$. 
However, in many situations it is more convenient to indicate the dependence on
the partition at the level of functions of colored graphs rather than in their
arguments. Accordingly, when the partition is clear from the context, we will
simply write $X$ (and similarly $G$ for the underlying graph).
When convenient, we will also omit the superscript $n$ in the notation
$B_i^{(n)}$, $i\in[k]$.
\end{remark}

We define the Gibbs probability density associated with \eqref{Hamilt_ERG2} by
$\mu^{(\bf{B})}_{n;\boldsymbol{\a},\bf{h}}=\frac{\exp{[\mathcal{H}^{(\bf{B})}_{n;\boldsymbol{\alpha},\bf{h}}}(X)]}{Z^{(\bf{B})}_{n;\boldsymbol{\alpha},\bf{h}}}$, which we view as a function on $\mathcal A^{(\mathbf B)}_n$.
The normalizing constant  $Z^{(\bf{B})}_{n;\alpha,\bf{h}}$, called {\em partition function}, is given by
\begin{equation}\label{pf}
Z^{(\bf{B})}_{n;\alpha,\bf{h}}:=\sum_{X \in \mathcal A^{(\mathbf B)}_n} \exp \left(\mathcal{H}^{(\bf{B})}_{n;\alpha,\bf{h}}(X)\right).
\end{equation}
We denote by $\mathbb P^{(\mathbf B)}_{n;\boldsymbol\alpha,\mathbf h}$ the corresponding
Gibbs measure, and by $\mathbb E^{(\mathbf B)}_{n;\boldsymbol\alpha,\mathbf h}$
the associated expectation.

\begin{remark}[Reduction to \ER model]\label{gen_Z}
Fix $\boldsymbol\alpha=0$.
\begin{enumerate}
\item If $k=1$ (so that $\mathbf h\equiv h\in\R$), the partition function \eqref{pf}
reduces to
\[
Z_{n;h}
=
\sum_{m=0}^{\binom{n}{2}}
\binom{\binom{n}{2}}{m}\,e^{hm},
\]
which corresponds to the classical dense Erdős-Rényi model with connection probability $\frac{e^{h}}{1+e^{h}}$.
\item If instead $k>1$, 
the partition function factorizes as
\begin{equation}\label{eq:Z_toy}
Z^{(\mathbf B)}_{n;\mathbf h}
=
\prod_{1\le i<j\le k}
\left(
\sum_{m=0}^{w_i w_j}
\binom{w_i w_j}{m}\,e^{h_{ij}m}
\right)
\prod_{i=1}^k
\left(
\sum_{m=0}^{\binom{w_i}{2}}
\binom{\binom{w_i}{2}}{m}\,e^{h_{ii}m}
\right),
\end{equation}
where $w_i:=|B^{(n)}_i|$ (see Assumption~\ref{ass:partition}). 
The binomial coefficients count the number of possible choices of 
inter and intra-community edges, respectively.
This extends the previous case to a dense inhomogeneous Erd\H{o}s--R\'enyi model
with connection probabilities $p_{ij}=\frac{e^{h_{ij}}}{1+e^{h_{ij}}}$, $i,j\in[k]$.\\
For instance, in the toy case $n=3$ and $k=2$, with $B_1=\{1,2\}$ and $B_2=\{3\}$,
Eq.~\eqref{eq:Z_toy} yields
\[
Z^{(\mathbf B)}_{3;\mathbf h}
=
\left(\sum_{m=0}^{2}\binom{2}{m}e^{h_{12}m}\right)
\left(\sum_{m=0}^{1}\binom{1}{m}e^{h_{11}m}\right),
\]
which collapses to $Z_{3;h}=1+ 3e^{h}+3e^{2h}+ e^{3h}$ when $k=1$.
\end{enumerate}
\end{remark}

\subsection{The space of colored graphons}\label{sec:color_Graphons}
In this section we introduce the main objects involved in our limiting results,
namely graphons and their colored counterparts.
The limit of a sequence of graphs can be represented by a measurable and symmetric function $g:[0,1]^2 \rightarrow [0,1]$, called {\em graphon}. The set of all graphons is denoted by $\mathcal{W}$.\\
Any sequence of graphs that converges in the appropriate way has a graphon as limit. Conversely, every graphon arises as the limit of an appropriate graph sequence.
\begin{remark}\label{empiric_graphon}
Any finite simple graph admits a graphon representation, called checkerboard graphon. 
\end{remark}
\begin{definition}[Checkerboard graphon]\label{def:checkg}
Let $H$ be a finite simple graph $H$ with vertex set $[m]$. The \emph{checkerboard} graphon $g^H$, corresponding to $H$, is defined by
\begin{equation}\label{graph_embedding}
g^H(x,y) = \left\{
\begin{array}{ll}
1 & \text{ if $\{\lceil mx \rceil, \lceil my \rceil\}$ is an edge in $H$}\\
0 & \text{ otherwise}
\end{array},
\right.
\end{equation}
where $(x,y) \in [0,1]^2$. In other words, $g^H$ is a step function corresponding to the adjacency matrix of $H$. This representation allows all simple graphs, regardless of the number of vertices, to be represented as elements of the  space $\mathcal{W}$.
\end{definition}

\begin{figure}[h!]
		\centering
		\includegraphics[scale=0.75]{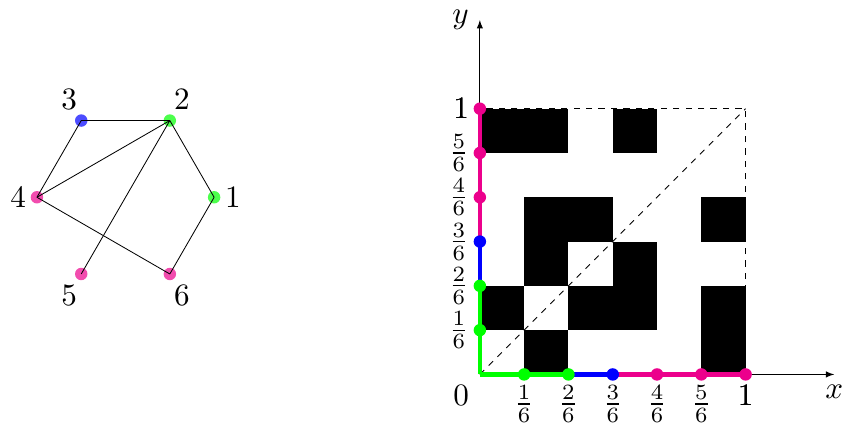}
		\caption{ \small Graph $H$ with $m=6$ vertices endowed with a partition into $k=3$ colors (left),
and the associated colored checkerboard graphon $g_{\mathcal{B}^{(3)}_n}^{H}$
(right).
		}
		\label{graphon}
	\end{figure}

Given a graph $H$, together with a partition $\mathbf B_n^{(k)}$ of the vertex
set, the associated checkerboard graphon $g^H$ naturally induces a partition
$\mathcal B_n^{(k)}$ of the unit interval $[0,1]$.
The pair $(g^H,\mathcal B_n^{(k)})$ will be referred to as a \emph{colored
checkerboard graphon} (see Fig.~\ref{graphon} for an illustration).
With this convention, there is a one-to-one correspondence between colored graphs
$(H,\mathbf B_n^{(k)})$ and colored checkerboard graphons $(g^H,\mathcal B_n^{(k)})$.
In the analysis of the limiting objects, however, we will work with the (fixed) partition
$\mathcal B^{(k)}$ introduced in Assumption~\ref{ass:partition}. We define the
space of $k$-colored graphons associated with $\mathcal B^{(k)}$ by
\begin{equation}\label{space_WB}
\mathcal W_{\mathcal B}^{(k)}
:=
\mathcal W \times \{\mathcal B^{(k)}\},
\end{equation}
equipped with the pseudo-metric $d_{\square}^{(k)}$  (the natural
extension of the cut distance to colored graphons, see \cite[Sec.~7]{GP}) given by
\begin{align}\label{cutcol}
&d_{\square}^{(k)}\bigl((g_1,\mathcal B^{(k)}),(g_2,\mathcal B^{(k)})\bigr)\notag\\
&\hspace{2cm}:=
\sup_{C,D\subseteq[0,1]}
\sum_{i,j\in[k]}
\left|
\int_{C\times D}
\mathds 1_{\mathcal B_i\times\mathcal B_j}(x,y)
\bigl(g_1(x,y)-g_2(x,y)\bigr)\,dx\,dy
\right|.
\end{align}
As for the standard homogeneous setting, one needs to take into account the arbitrary labeling of vertices, as they are embedded in the unit interval. To this aim, we introduce the following equivalence relation on $\mathcal{W}_{\mathcal{B}}^{(k)}$.
The idea is that two colored graphons are equivalent if they represent the same object up to a
measure-preserving relabeling of $[0,1]$ that does not mix the color classes. In other words, we are
allowed to reparametrize the underlying space, but only within each block of the fixed partition
$\mathcal{B}^{(k)}=(\mathcal B_1,\dots,\mathcal B_k)$, so that the coloring is kept unchanged (up to null sets).

\begin{definition}
Let $\Sigma_{\mathcal B}$ denote the set of all measure-preserving maps
$\sigma:[0,1]\to[0,1]$ such that
\[
\lambda\bigl(\sigma^{-1}(\mathcal B_i)\,\triangle\,\mathcal B_i\bigr)=0
\qquad\text{for all } i\in[k]\footnote{Given two sets $A$ and $B$, \(A \triangle B := (A \setminus B) \cup (B \setminus A)\) denotes the symmetric difference.},
\]
(recall that $\lambda(\cdot)$ denotes the Lebesgue measure).
For $\sigma\in\Sigma_{\mathcal B}$, we define
\[
(g,\mathcal B^{(k)})^\sigma := (g^\sigma,\mathcal B^{(k)}),
\qquad
g^\sigma(x,y):=g(\sigma(x),\sigma(y)).
\]
\end{definition}

\begin{definition}[Equivalence relation] \label{def:eq_Rel}
Two colored graphons
$(g_1,\mathcal B^{(k)})$, \newline $(g_2,\mathcal B^{(k)})\in\mathcal W_{\mathcal B}^{(k)}$
are said to be equivalent, and we write
$(g_1,\mathcal B^{(k)})\sim(g_2,\mathcal B^{(k)})$,
if there exists $\sigma\in\Sigma_{\mathcal B}$ such that
\[
(g_2,\mathcal B^{(k)})=(g_1,\mathcal B^{(k)})^\sigma.
\]
\end{definition}

The quotient space under $\sim$ is denoted by $\widetilde{\mathcal{W}}^{(k)}_{\mathcal{B}}$, and $\tau:(g,\mathcal{B}^{(k)}) \mapsto \widetilde{(g,\mathcal{B}^{(k)})}$ is the natural mapping associating a colored graphon with its equivalence class. 
Incorporating the equivalence relation $\sim$ in \eqref{cutcol} yields a distance
\begin{equation}\label{ct2}
\delta^{(k)}_\square\bigl(
\widetilde{(g_1,\mathcal B^{(k)})},
\widetilde{(g_2,\mathcal B^{(k)})}
\bigr)
=
\inf_{\sigma\in\Sigma_{\mathcal B}}
d_\square^{(k)}\bigl(
(g_1,\mathcal B^{(k)}),
(g_2,\mathcal B^{(k)})^\sigma
\bigr).
\end{equation}
\begin{figure}[h!]
		\centering
		\includegraphics[scale=0.75]{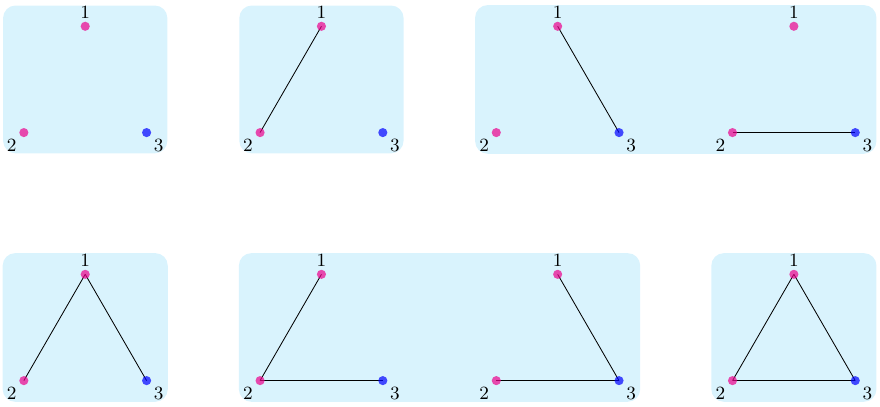}
		\caption{Example of a partition into equivalence classes (light-blue shaded) when $k=2$ and $n=3$.
		}
		\label{fig_equiv_classes}
	\end{figure}
Now, we set $g_{\mathcal{B}}:= (g,\mathcal{B}^{(k)})$ (and similarly $g^{G}_{\mathcal{B}_n}:= (g,\mathcal{B}_n^{(k)})$ for the colored checkerboard graphon) and we denote by $\tilde{g}^{(\mathcal{B})}$ its equivalence class (see Fig.~\ref{fig_equiv_classes} for an example of a partition into equivalence classes). 

From the fact that the graphon quotient space
$
(\widetilde{\mathcal W},\delta_\square)
$
is compact (\cite[Thm.~5.1]{LS07}), where $\delta_{\square}$ denotes the usual cut metric (\cite[Sec.~3]{LS07}), we deduce the following. 
\begin{lemma}[Compactness] \label{compact}
The space $(\widetilde{W}_{\mathcal{B}}^{(k)},\delta^{(k)}_{\square})$ is compact. 
\end{lemma}
\begin{proof}
Let 
$
F:\widetilde{W}^{(k)}_{\mathcal B}
\longrightarrow
(\widetilde{\mathcal W})^{k^2},
$
be the map that associates to the equivalence class of a pair
\((g,\mathcal B^{(k)})\)
the $k^2$-tuple of its block graphons \footnote{Given $g\in\mathcal W$, one can define them as 
$
G_{ij}(s,t):=g(\phi_i(s),\phi_j(t)),
$
$s,t\in[0,1]$, $i,j\in[k]$,
where $\phi_i$ is the measure-preserving bijection $\phi_i:([0,1],\lambda)\to(\mathcal B_i,\lambda_i)$, and $\lambda_i:=\lambda(\cdot)/b_i$ is the normalized
probability measure on $\mathcal B_i$.
} (modulo the usual relabeling on $[0,1]$ in each coordinate). We equip $(\widetilde{\mathcal W})^{k^2}$ with the weighted product metric
induced by $\delta_\square$ and the block weights $(b_i)_{i\in[k]}$.
One can readily verify that this map is an isometric embedding onto its image. Since $(\widetilde{\mathcal W},\delta_\square)$ is compact,
its finite product $(\widetilde{\mathcal W})^{k^2}$ is compact as well,
and hence complete.
As $F$ is an isometric embedding,
its image $F(\widetilde W_{\mathcal B}^{(k)})$ is complete.
Therefore it is closed in $(\widetilde{\mathcal W})^{k^2}$.
Being a closed subset of a compact space, it is compact.
\end{proof}

The last notion we need to introduce is the \emph{cell-restricted subgraph density}. This is just the restriction of the subgraph density to the portion of the graphon whose arguments (nodes, if we think to the discrete setting) belong to some specific elements of the partition $\mathcal{B}^{(k)}$. 

\begin{definition}[Cell-restricted subgraph density]\label{cell_r_sd}
Let $H$ be a finite simple graph with vertex set $V(H)=[m]$ and edge set
$\mathcal E(H)$.
Fix a vector of colors $\boldsymbol\ell=(\ell_1,\dots,\ell_m)\in[k]^m$.
For a colored graphon $g_{\mathcal B}=(g,\mathcal B^{(k)})\in
\mathcal W^{(k)}_{\mathcal B}$, define
\begin{equation}\label{def_graphon_hom_density}
t^{(\mathcal B)}_{\boldsymbol\ell}(H,g)
:=
\int_{\mathcal B_{\ell_1}\times\cdots\times\mathcal B_{\ell_m}}
\prod_{\{i,j\}\in\mathcal E(H)} g(x_i,x_j)\,
dx_1\cdots dx_m .
\end{equation}
\end{definition}
Note that $t^{(\mathcal B)}_{\boldsymbol\ell}(H,g)$ is not normalized to take values
in $[0,1]$, since no normalization by the Lebesgue measure of the cells is included. Building on the notion of cell-restricted subgraph densities, we introduce the entropy and interaction functionals that enter the large deviation rate function and the associated variational principle in Thm.~\ref{thm_LLN} below.

\begin{definition}[Entropy and interaction functionals]\label{def:I-U}
For $u\in[0,1]$ set
\[
I(u):=u\ln u + (1-u)\ln(1-u),
\]
extended continuously to $[0,1]$ by $I(0)=I(1)=0$.
For $\widetilde g^{(\mathcal{B})}\in\widetilde{\mathcal W}^{(k)}_{\mathcal B}$, let $g$ denote the first component of any representative
$g_{\mathcal B}=(g,\mathcal B^{(k)})$ of $\widetilde g^{(\mathcal{B})}$. We define:
\begin{enumerate}
\item the (block) entropy functional
\begin{equation}\label{eq:def-entropy-functional}
\mathcal I^{(\mathcal B)}(\widetilde g)
:=
\sum_{i,j\in[k]}\int_{\mathcal B_i\times\mathcal B_j} I\bigl(g(x,y)\bigr)\,dx\,dy;
\end{equation}

\item the interaction functional 
\begin{equation}\label{eq:U-functional}
U^{(\mathcal B)}_{\boldsymbol{\alpha},\mathbf h}(\widetilde g)
:=
\frac16\sum_{i,j,\ell\in[k]}
\alpha_{ij\ell}\, t^{(\mathcal B)}_{ij\ell}(H_2, g)
+
\frac12\sum_{i,j\in[k]}
h_{ij}\, t^{(\mathcal B)}_{ij}(H_1, g).
\end{equation}

\end{enumerate}
\end{definition}


\begin{remark}\label{remark_funz}
Although the functional $\mathcal I^{(\mathcal B)}$ in \eqref{eq:def-entropy-functional} does not genuinely
depend on the limiting partition, since the latter can be absorbed into the
integration domains,  we adopt this notation to emphasize that
$\mathcal I^{(\mathcal B)}$ is naturally viewed as a functional acting on the pair
$(\widetilde W,\mathcal B^{(k)})$.
\end{remark}

\subsection{Main results}
\label{sec:main_res}
Let
\(
(\mathbb{P}^{({\bf B})}_{n;\boldsymbol{\alpha},\mathbf h})_{n\ge1}
\)
be the sequence of Gibbs probability measures associated with the
Hamiltonian~\eqref{Hamilt_ERG2}.
We denote by
\(
\widetilde{\mathbb{P}}^{({\mathcal B})}_{n;\boldsymbol{\alpha},\mathbf h}
\)
their push-forward onto the quotient space
\(
(\widetilde{\mathcal W}^{(k)}_{\mathcal B},\delta^{(k)}_\square),
\)
 obtained via
the mapping
\begin{equation}\label{map_pushforward}
G \longmapsto g_{\mathcal B_n}^G
\longmapsto \tau\!\left(\mathfrak c_n(g_{\mathcal B_n}^G)\right),
\end{equation}
where \(\mathfrak c_n:\mathcal W^{(k)}_{\mathcal B_n}\to\mathcal W^{(k)}_{\mathcal B}\)
is the `color blind' map replacing the finite partition \(\mathcal B_n\) with
the limiting partition \(\mathcal B\), and \(\tau\) denotes the canonical
projection onto equivalence classes introduced above \eqref{ct2}.

\begin{theorem}[LDP and limiting free energy]\label{thm:LDP-free-energy}
The sequence $\bigl(\widetilde{\P}^{(\mathcal B)}_{n;\boldsymbol{\alpha},\mathbf h}\bigr)_{n\ge1}$
satisfies a large deviation principle on
$(\widetilde{\mathcal W}^{(k)}_{\mathcal B},\delta^{(k)}_\square)$
with speed $n^2$ and good rate function
\begin{equation}\label{eq:LDP-rate}
\mathcal I^{(\mathcal B)}_{\boldsymbol{\alpha},\mathbf h}(\widetilde f)
:=
\frac{1}{2}\mathcal I^{(\mathcal B)}(\widetilde f)
-
U^{(\mathcal B)}_{\boldsymbol{\alpha},\mathbf h}(\widetilde f)
-
\inf_{\widetilde g\in\widetilde{\mathcal W}^{(k)}_{\mathcal B}}
\Bigl\{
\frac{1}{2}\mathcal I^{(\mathcal B)}(\widetilde g)
-
U^{(\mathcal B)}_{\boldsymbol{\alpha},\mathbf h}(\widetilde g)
\Bigr\}.
\end{equation}
Moreover, the limiting free energy is given by 
\begin{equation}\label{varpr}
f^{(\mathcal{B})}_{\boldsymbol{\alpha}, {\bf{h}}}:=\lim_{n\to\infty}\frac{1}{n^2}\ln Z^{({\bf  B})}_{n;\boldsymbol{\alpha},\mathbf h}
=
\sup_{\widetilde g\in\widetilde{\mathcal W}^{(k)}_{\mathcal B}}
\Bigl\{
U^{(\mathcal B)}_{\boldsymbol{\alpha},\mathbf h}(\widetilde g)
-
\frac{1}{2}\mathcal I^{(\mathcal B)}(\widetilde g)
\Bigr\}.
\end{equation}
\end{theorem}

\begin{theorem}[Euler--Lagrange equations]\label{thm:EL-alphaijk}
Let $\widetilde{g}^{\star(\mathcal{B})}\in \widetilde{\mathcal W}^{(k)}_{\mathcal B}$ be a maximizer of
the functional in \eqref{varpr}.
Then, for any representative element $g_{\mathcal B}=(g,\mathcal B^{(k)})\in \widetilde{g}^{\star(\mathcal{B})}$ the following holds.
\begin{enumerate}
\item \label{item:el:interior} There exists 
$\delta>0$ such that
\begin{equation}\label{eq:interior-alphaijk}
\delta\le g(x,y)\le 1-\delta\qquad \text{for a.e.\ }(x,y)\in[0,1]^2.
\end{equation}
\item \label{item:el:eq}
For a.e.\ $(x,y)\in[0,1]^2$,
\begin{equation}\label{eq:logistic-pointwise-alphaijk}
g(x,y)=
\frac{\exp\!\left(h_{\mathcal B}(x,y)+(\mathcal T^{(\mathcal{B})}_{\boldsymbol\alpha}g)(x,y)\right)}
{1+\exp\!\left(h_{\mathcal B}(x,y)+(\mathcal T^{(\mathcal{B})}_{\boldsymbol\alpha}g)(x,y)\right)},
\end{equation}
where $\mathcal T^{(\mathcal{B})}_{\boldsymbol\alpha}$ is the (nonlinear) triangle operator
\begin{equation}\label{eq:Aalpha-def}
\bigl(\mathcal T^{(\mathcal{B})}_{\boldsymbol\alpha}g\bigr)(x,y)
:=
\int_0^1 \Big(
\Delta_{\boldsymbol\alpha; \mathcal{B}}(x,y,z)\,g(x,z)\,g(y,z)
\Big)\,dz,
\end{equation}
with 
\begin{equation}\label{eq:aalpha-kernel}
\Delta_{\boldsymbol\alpha; \mathcal{B}}(x,y,z)
:=
\sum_{i,j,\ell=1}^k \alpha_{ij\ell}\,
\mathds 1_{\mathcal B_i}(x)\,\mathds 1_{\mathcal B_j}(y)\,\mathds 1_{\mathcal B_\ell}(z),
\end{equation}
and 
\begin{equation}\label{eq:hB}
h_{\mathcal B}(x,y):=\sum_{i,j=1}^k h_{ij}\,\mathds 1_{\mathcal B_i}(x)\,\mathds 1_{\mathcal B_j}(y).
\end{equation}
\end{enumerate}
\end{theorem}

\begin{theorem}[Scalar problem] \label{thm:scalar_probl}
Suppose that $\boldsymbol{\alpha}\geq 0$. Define the set of admissible symmetric matrices
\begin{equation}\label{def:c_sym}
\mathsf{C}_{\mathrm{sym}}
:=\Bigl\{C=(c_{ij})_{i,j\in[k]}\in[0,1]^{k\times k}:\ c_{ij}=c_{ji}\ \ \forall i,j\in[k]\Bigr\},
\end{equation}
and for $C\in\mathsf{C}_{\mathrm{sym}}$ let $g_C$ be the corresponding $\mathcal B$--block graphon
\begin{equation}\label{eq:gC-def}
g_C(x,y):=\sum_{i,j=1}^k c_{ij}\,\mathds 1_{\mathcal B_i}(x)\mathds 1_{\mathcal B_j}(y),\qquad (x,y)\in[0,1]^2.
\end{equation}
Then:
\begin{enumerate}
\item \label{scal:item1} Every maximizer $\widetilde{g}^{\star(\mathcal{B})}\in \widetilde{\mathcal W}^{(k)}_{\mathcal B}$ of the variational problem in
\eqref{varpr} admits a representative of the form $(g_{C}, \mathcal{B}^{(k)})$. In particular,
the variational problem reduces to a finite-dimensional one:
\begin{equation}\label{eq:finite-dim-problem}
\sup_{C\in\mathsf{C}_{\mathrm{sym}}}
\left\{
\frac{1}{6}\sum_{i,j,\ell=1}^k \alpha_{ij\ell}\,t^{(\mathcal B)}_{ij\ell}(H_2,g_C)
+\frac12\sum_{i,j=1}^k h_{ij}\, t^{(\mathcal B)}_{ij}(H_1,g_C)
-\frac12\,\mathcal I^{(\mathcal B)}(g_C)
\right\}.
\end{equation}
\item \label{scal:item2} There exists at least one maximizer
$
C^\star \equiv C^\star(\boldsymbol\alpha,\mathbf h)\in\mathsf{C}_{\mathrm{sym}}
$
of \eqref{eq:finite-dim-problem}.
\item \label{scal:item3} Any maximizer $C^\star=(c^\star_{ij})_{i,j\in[k]}$
satisfies the Euler--Lagrange
fixed-point system: for all $i,j\in[k]$,
\begin{equation}\label{eq:fixed-point-scalar}
c^{\star}_{ij}
=
\frac{\exp\!\left(
h_{ij}+\sum_{\ell=1}^k b_\ell\,c^{\star}_{i\ell}c^{\star}_{j\ell}\alpha_{ij\ell}
\right)}
{1+\exp\!\left(
h_{ij}+\sum_{\ell=1}^k b_\ell\,c^{\star}_{i\ell}c^{\star}_{j\ell}\,
\alpha_{ij\ell})
\right)}
\end{equation}
(see Fig.~\ref{lim_graphon} for an illustration of $(g_{C^{\star}},\mathcal{B}^{k})$).
\end{enumerate}
\end{theorem}
\begin{figure}[h!]
		\centering
		\includegraphics[scale=0.85]{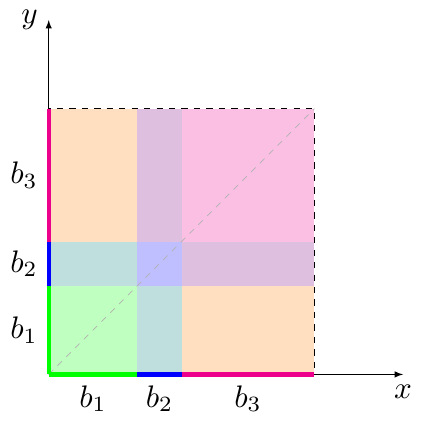}
		\caption{ Illustration of a maximizer $(g_{C^{\star}},\mathcal{B}^{k})$ of \eqref{varpr}, where $C^{\star}$ maximizes the scalar problem \eqref{eq:finite-dim-problem} and satisfies the system in \eqref{eq:fixed-point-scalar}.
		}
		\label{lim_graphon}
	\end{figure}
\begin{theorem}[Contraction map]\label{prop:small-alpha-tensor}
Let $k\in\N$,  $\mathbf h=(h_{ij})_{i,j\in[k]}\in\R^{k\times k}$, and 
$\boldsymbol\alpha=(\alpha_{ij\ell})_{i,j,\ell\in[k]}\in\R^{k\times k\times k}$.
Define the logistic function $\sigma:\R\to(0,1)$ by
\[
\sigma(x):=\frac{e^x}{1+e^x}.
\]
We consider the map
\[
S^{(\mathcal B)}_{\boldsymbol\alpha;\mathbf h}:[0,1]^{k\times k}\to(0,1)^{k\times k}
\]
defined componentwise, for $C=(c_{ij})_{i,j\in[k]}\in\mathsf{C}_{\mathrm{sym}}$, by
\begin{equation}\label{eq:Smap-def}
\bigl(S^{(\mathcal B)}_{\boldsymbol\alpha;\mathbf h}(C)\bigr)_{ij}
:=
\sigma\!\left(
h_{ij}
+\sum_{\ell=1}^k b_\ell\,c_{i\ell}c_{j\ell}\,
\alpha_{ij\ell}
\right),
\qquad i,j\in[k].
\end{equation}
Then $S^{(\mathcal B)}_{\boldsymbol\alpha;\mathbf h}$ is globally Lipschitz on
$[0,1]^{k\times k}$ with respect to $\|\cdot\|_\infty$, with Lipschitz constant
\begin{equation}\label{eq:Lipschitz-constant-tensor}
L:=\frac12\|\boldsymbol\alpha\|_\infty,
\end{equation}
where $\|\boldsymbol\alpha\|_\infty
:=
\max_{i,j,\ell\in[k]}|\alpha_{ij\ell}|$.
If $L<1$ (i.e., if $\|\boldsymbol\alpha\|_\infty<2$), then
$S^{(\mathcal B)}_{\boldsymbol\alpha;\mathbf h}$ is a contraction on
$([0,1]^{k\times k},\|\cdot\|_\infty)$ and therefore admits a unique fixed point
$C^\star\equiv C^\star(\boldsymbol\alpha,\mathbf h)$.
\end{theorem}
Before moving to the proofs, some remarks are in order. 
\begin{remark}
When $\boldsymbol{\alpha}\equiv\alpha\in\mathbb{R}$,
the condition $\|\boldsymbol{\alpha}\|_\infty<2$
reduces to the scalar constraint $|\alpha|<2$, which coincides with the range of validity
of \cite[Thm.~6.2]{CD} for the standard edge-triangle model.
The region $|\alpha|<2$ is sometimes referred to as Dobrushin's uniqueness region. 
\end{remark}

\begin{remark}\label{rem:logic}
The role of Thms.~\ref{thm:EL-alphaijk}--\ref{thm:scalar_probl} can be summarized as follows.
The Euler--Lagrange equations characterize maximizers of the variational problem
through a pointwise nonlinear equation and imply that any maximizer takes values
strictly between $0$ and $1$.
Under the additional assumption $\boldsymbol\alpha\ge 0$, the problem
admits a maximizer that is $\mathcal B$--block constant, so that the infinite-dimensional
variational problem reduces to a finite-dimensional one in terms of matrices
$C=(c_{ij})_{i,j\in[k]}$.\\
The fixed-point map $S^{(\mathcal B)}_{\boldsymbol\alpha;\mathbf h}$ is introduced in Thm.~\ref{prop:small-alpha-tensor}
to analyze the corresponding Euler--Lagrange system.
For sufficiently small $\boldsymbol\alpha$, this map is a contraction on the
complete space $[0,1]^{k\times k}$ and therefore admits at most one fixed point.
\end{remark}
Combining these results, one can readily deduce 
the following corollary.

\begin{corollary}[Uniqueness of the maximizer]\label{cor:unique-maximiser}
Assume that $0\leq \|\boldsymbol\alpha\|_\infty<2$.
Then the variational problem \eqref{varpr} admits a \emph{unique} maximizer
$\widetilde g^{\star(\mathcal B)}\in\widetilde{\mathcal W}^{(k)}_{\mathcal B}$.
Moreover, any representative of $\widetilde g^{\star(\mathcal B)}$ has the form $g_{C^\star}$ as in \eqref{eq:gC-def}, where
$C^\star\equiv C^\star(\boldsymbol\alpha,\mathbf h)\in \mathsf{C}_{\mathrm{sym}}$ is the unique solution of the finite-dimensional Euler--Lagrange
system \eqref{eq:fixed-point-scalar}.
\end{corollary}

Having established a uniqueness result, it is natural to turn to the question of whether a law of large numbers holds for subgraph densities. We present our results for the edge density; however, we claim that the same argument applies to any subgraph density.
In the results below, we keep the dependence on \(n\) inside the partition explicit, in order to avoid any confusion with the measure in Eq.~\eqref{LLN} (arising from Rem.~\ref{rem:kolmogorov-block}).
Let \(E_n\) denote the random number of edges of a ERG sampled according to the probability measure \(\mathbb{P}^{(\mathbf{B}_n)}_{n;\boldsymbol{\alpha},\mathbf{h}}\).
The following holds.

\begin{theorem}[Exponential convergence for $E_n$]\label{thm_ExpC}
Suppose $0\le \|\boldsymbol\alpha\|_\infty<2$. Then,
\begin{equation}\label{exp_conv}
\frac{2E_n}{n^2}\xrightarrow{\ \mathrm{exp}\ }
\sum_{i,j=1}^k b_i b_j\,c^\star_{ij}, \quad \text{ w.r.t. } \mathbb{P}^{({\bf{B}_n})}_{n;\boldsymbol{\a},\bf{h}}, \text{ as } n \to +\infty,
\end{equation}	
where, for each $i,j\in[k]$, $c^{\star}_{i,j}$ solves \eqref{eq:fixed-point-scalar}. 
\end{theorem}

\begin{remark}[Kolmogorov extension]\label{rem:kolmogorov-block}
For $n\in\N$, recall that
$
\mathcal A_n=\{0,1\}^{\binom{n}{2}}$
is the set of adjacency matrices of simple graphs on $[n]=\{1,\dots,n\}$.
For $m>n$, denote by
\[
\pi_{m\to n}:\mathcal A_m\to\mathcal A_n
\]
the canonical projection obtained by restricting a configuration to the
induced subgraph on $[n]$.
Assume the projective consistency condition
\begin{equation}\label{eq:projective-consistency}
\mathbb P^{(\mathbf B_{n+1})}_{n+1;\boldsymbol\alpha,\mathbf h}
\circ \pi_{n+1\to n}^{-1}
=
\mathbb P^{(\mathbf B_n)}_{n;\boldsymbol\alpha,\mathbf h}
\qquad\text{for all }n\in\N,
\end{equation}
that is, under
$\mathbb P^{(\mathbf B_{n+1})}_{n+1;\boldsymbol\alpha,\mathbf h}$
the induced subgraph on $[n]$ has law
$\mathbb P^{(\mathbf B_n)}_{n;\boldsymbol\alpha,\mathbf h}$.
Then, by the Kolmogorov extension theorem, there exists a unique probability measure
$
\mathbb P^{(\mathbf B)}_{\boldsymbol\alpha,\mathbf h}
$
on the infinite product space
$\bigl(\{0,1\}^{\binom{\N}{2}},\mathscr F\bigr)$
whose marginals on $\mathcal A_n$ coincide with
$\mathbb P^{(\mathbf B_n)}_{n;\boldsymbol\alpha,\mathbf h}$ for all $n$.
Here $\mathscr F$ denotes the Borel $\sigma$-algebra and
$\mathbf B=(B_i)_{i\in[k]}$ is a deterministic partition of $\N$
such that $B_n^{(i)}=B_i\cap[n]$ for each $i\in[k]$ (where $B_n^{(i)}$ are the elements of the partition given in Assumption~\ref{ass:partition}).
\end{remark}

\begin{remark}
In view of Rem.~\ref{rem:kolmogorov-block}
and Assumption~\ref{ass:partition},
the partition $\mathbf B$ of $\mathbb N$
is such that for all $i\in[k]$,
\[
\frac{|B_i\cap[n]|}{n} \to b_i
\quad \text{as } n\to\infty,
\]
where $b_i=\lambda(\mathcal B_i)$.
\end{remark}

Finally, Thm.~\ref{thm_LLN} together with Rem.~\ref{rem:kolmogorov-block} leads to the following.

\begin{theorem}[SLLN for $E_n$]\label{thm_LLN}
Suppose $0\le \|\boldsymbol\alpha\|_\infty<2$ and that the finite-size partition ${\bf{B}}_n$ satisfies \eqref{eq:projective-consistency}. Then,
\begin{equation}\label{LLN}
\frac{2E_n}{n^2}\xrightarrow{\ \mathrm{a.s.}\ }
\sum_{i,j=1}^k b_i b_j\,c^\star_{ij}, \quad \text{ w.r.t. } \mathbb{P}^{({\bf{B}})}_{\boldsymbol{\a},\bf{h}}, \text{ as } n \to +\infty,
\end{equation}	
where, for each $i,j\in[k]$, $c^{\star}_{i,j}$ solves \eqref{eq:fixed-point-scalar}, and $\mathbb{P}^{({\bf{B}})}_{\boldsymbol{\a},\bf{h}}$ is the infinite-volume Gibbs measure
constructed in Rem.~\ref{rem:kolmogorov-block}. 
\end{theorem}

\section{Proofs} \label{sec_proofs}
\subsection{LDP}
\begin{proof}[Proof of Thm.~\ref{thm:LDP-free-energy} (LDP)]
The idea of the proof is to represent
\(\mathbb{P}^{({\bf B}_n)}_{n;\boldsymbol{\alpha},\mathbf h}\)
as a tilted measure with respect to the Erd\H{o}s--R\'enyi reference measure.
To this end, we first derive an alternative representation of the Hamiltonian
\begin{align*}
\mathcal{H}^{(\mathbf B_n)}_{n;\boldsymbol\alpha,\mathbf h}(X)
&= \frac{1}{n}\sum_{i,j,\ell\in[k]} \alpha_{ij\ell}
\sum_{\substack{u\in B_i,\ v\in B_j,\ w\in B_\ell \\ u<v<w}}
X_{uv}\,X_{vw}\,X_{uw}\\
&\hspace{4cm}+\sum_{i,j \in [k]} h_{ij}\sum_{\substack{u\in B_i,\ v\in B_j\\ u<v}} X_{uv}\notag
\end{align*}
in terms of cell-restricted subgraph densities
(see Def.~\ref{cell_r_sd}).
Now, the partition ${\bf{B}}_n\equiv \mathbf B^{(k)}_{n}$,
in turn induces a partition
\(\mathcal B^{(k)}_n=(\mathcal B^{(n)}_1,\dots,\mathcal B^{(n)}_k)\) of \([0,1]\)
into consecutive intervals of lengths
\(|\mathcal B^{(n)}_i|=|B^{(n)}_i|/n\), for all \(i=1,\dots,k\).
Let \(H_1\) and \(H_2\) denote, respectively, the edge and
triangle subgraphs.
It is easy to check that
\begin{align}
\frac{2\sum_{\substack{u\in B_i,\ v\in B_j\\ u<v}} X_{uv}}{n^2}&= t^{({\mathcal{B}}_{n})}_{ij}(H_1, g^{X}) \\
\frac{6\sum_{\substack{u\in B_i,\ v\in B_j,\ w\in B_\ell \\ u<v<w}}
X_{uv}\,X_{vw}\,X_{uw}}{n^3}&= t^{({\mathcal{B}}_{n})}_{ij\ell}(H_2, g^{X}) \label{tr_subd},
\end{align}
where we use the checkerboard representation introduced in
Def.~\ref{def:checkg}, identifying the adjacency matrix $X$
with the corresponding graph.
For notational simplicity, we write $i,j,\ell$ in place of
$\ell_1,\ell_2,\ell_3$.
Consequently, we may rewrite (dropping the superscript $k$)
\begin{equation}\label{ident:HU}
\mathcal{H}^{(\mathbf B_n)}_{n;\boldsymbol\alpha,\mathbf h}(X)
=n^2U^{(\mathcal{B}_n)}_{\boldsymbol\alpha,\mathbf h}(X)
\end{equation}
where 
\begin{equation}\label{def:U}
U^{(\mathcal{B}_n)}_{\boldsymbol\alpha,\mathbf h}(X) =\frac{1}{6} \sum_{i,j,\ell\in [k]}\alpha_{ij\ell}\; t^{(\mathcal{B}_n)}_{ij\ell}(H_2, g^{X}) + \frac{1}{2}\sum_{i,j \in [k]} h_{ij} \;t^{(\mathcal{B}_n)}_{ij}(H_1, g^{X}).
\end{equation}

\begin{remark}
Note that
\[
t^{(\mathcal B_{n})}_{ij}(H_1,g^X)
=\frac{|B^{(n)}_i||B^{(n)}_j|}{n^2}\,
  \frac{2\sum_{u\in B^{(n)}_i,\; v\in B^{(n)}_j,\;u<v}X_{uv}}
       {|B^{(n)}_i||B^{(n)}_j|},
\quad i,j\in[k].
\]
Thus, the first factor is the Lebesgue measure of the block
$\mathcal B^{(n)}_i\times\mathcal B^{(n)}_j\subset[0,1]^2$, while the second
factor is the edge homomorphism density\footnote{See, e.g., \cite[Eq.~(2.9)]{CD} for the definition.} of the discrete subgraph
induced by $B^{(n)}_i$ and $B^{(n)}_j$. The same holds for the term $t^{({\mathcal{B}}_{n})}_{ij\ell}(H_2, g^{X})$ in \eqref{tr_subd}.
\end{remark}

This representation shows that the Hamiltonian extends naturally to
$\mathcal W^{(k)}_{\mathcal B_n}$.
We now replace the finite partition $\mathcal B_n$ with the limiting
partition $\mathcal B$ introduced in
Assumption~\ref{ass:partition}.
To quantify the discrepancy, define the error term
\[
R_{n;\boldsymbol\alpha,\mathbf h}(X)
:=
U^{(\mathcal B_n)}_{\boldsymbol\alpha,\mathbf h}(X)
-
U^{(\mathcal B)}_{\boldsymbol\alpha,\mathbf h}(X),
\]
where $U^{(\mathcal P)}_{\boldsymbol\alpha,\mathbf h}$ is given in
\eqref{def:U} for $\mathcal P\in\{\mathcal B,\mathcal B_n\}$.
In view of \eqref{ident:HU}, the Hamiltonian can then be decomposed as
\begin{align}
\mathcal{H}^{(\mathbf B_n)}_{n;\boldsymbol\alpha,\mathbf h}(X)
&=
n^2\Bigl(
U^{(\mathcal B)}_{\boldsymbol\alpha,\mathbf h}(X)
+
R_{n;\boldsymbol\alpha,\mathbf h}(X)
\Bigr).
\label{eq:U_graphon}
\end{align}

We will show that $R_{n;\boldsymbol\alpha,\mathbf h}(X)=o(1)$,
so that its contribution to the Hamiltonian is of order $o(n^2)$.
This is established in the following lemma, whose proof is deferred to
the end of the paragraph.

\begin{lemma}\label{lem:r1-bound}
Recall Assumption~\ref{ass:partition} with the condition imposed in \eqref{eq:block-proportions}. 
Define
\begin{equation}\label{def:eta}
\eta_n({\bf{B}}) := \max_{1\le i\le k} \left| \frac{w_i}{n} - b_i\right|.
\end{equation}
For each graph $X\in \mathcal{A}^{(\bf{B})}_{n}$, set
\begin{align}
r^{(1)}_{ij;n}(X)
&:= t^{(\mathcal B)}_{ij}(H_1,g^X)
   - t^{(\mathcal B_{n})}_{ij}(H_1,g^X)
\quad i,j\in[k] \label{eq:r1}\\
r^{(2)}_{ij\ell;n}(X)
&:= t^{(\mathcal B)}_{ij\ell}(H_2,g^X)
   - t^{(\mathcal B_{n})}_{ij}(H_1,g^X)
\quad i,j,\ell\in[k].\label{eq:r2}
\end{align}
Then there exists constants $\gamma^{(1)}_k := 4k-2$ and $\gamma_k^{(2)}:=6k-3$  such that
\begin{align}
\bigl|r^{(1)}_{ij,n}(X)\bigr|
&\le \gamma^{(1)}_k\,\eta_n({\bf{B}})
\qquad \label{ineq:r1}\\
\bigl|\,r^{(2)}_{ijq,n}(X)\,\bigr|
&\le\; \gamma_k^{(2)}\,\eta_n({\bf{B}}) \label{ineq:r2}
\end{align}
for all $X\in\mathcal{A}^{(\bf{B})}_{n}$,  $i,j,\ell\in[k]$ and $n\ge1$.
\end{lemma}
\begin{remark}[Uniform bound of the Hamiltonian error]
As a consequence of \eqref{ineq:r1}--\eqref{ineq:r2} and \eqref{eq:block-proportions}, we obtain the bound:
\begin{equation}\label{ineq:boundR}
\sup_{X\in\mathcal{A}^{({\bf{B}})}_n}
\bigl|R_{n;\boldsymbol{\alpha},\bf {h}}(X)\bigr|
\;\le\;
\gamma\,\eta_n({\bf{B}})=o(1),
\end{equation}
where $\gamma\equiv \gamma(\alpha_{\infty},h_{\infty},k)$ depends on  $\alpha_{\infty}:=||\boldsymbol{\alpha}||_{\infty}$, $h_{\infty}:=||{\bf{h}}||_{\infty}$ and $k$.
\end{remark}
As a second step, observe that
$
t^{(\mathcal B)}_{\boldsymbol\ell}(H,g)
=
t^{(\mathcal B)}_{\boldsymbol\ell}(H,\tilde g).
$
Consequently, the Hamiltonian admits a well-defined extension to
$\widetilde{\mathcal W}_{\mathcal B}^{(k)}$.
Recalling \eqref{pf} and \eqref{ident:HU}, we may therefore express the
partition function as follows:
\begin{align}
Z^{({\bf{B}}_n)}_{n;\boldsymbol{\alpha},\bf{h}}&=
\sum_{X\in \mathcal{A}^{(\bf{B})}_{n}}\exp{[n^2U^{(\mathcal{B}_n)}_{\boldsymbol{\alpha}, \bf{h}}(X)]}\\
&=\sum_{X\in \mathcal{A}^{(\bf{B})}_{n}}\exp{[n^2U^{(\mathcal{B})}_{\boldsymbol{\alpha}, \bf{h}}(X)]}\exp{[n^2R_{n;\boldsymbol{\alpha}, \bf{h}}(X)]}.\label{eq:part_fct_sum1}
\end{align}
We now seek a convenient way to decompose the sum above.
First, we recall the map \eqref{map_pushforward}.
We use the shorthand notation
$
g^{X}_{\mathcal B}:=\mathfrak c_n\bigl(g_{\mathcal B_n}^X\bigr)
$
for the checkerboard graphon associated with $X$,
equipped not with its natural partition $\mathcal B_n$, but with the limiting
partition $\mathcal B$.
We then rewrite the sum in \eqref{eq:part_fct_sum1} as follows:

\begin{align}
Z^{(\mathbf B_n)}_{n;\boldsymbol{\alpha},\mathbf h}
&=\sum_{\tilde g\in \widetilde{\mathcal W}^{(k)}_{\mathcal B}}\sum_{X\in \mathcal{A}^{(\bf{B})}_{n}: g^{X}_{\mathcal{B}}\in \tilde{g}}
\exp\Bigl[n^2 U^{(\mathcal B)}_{\boldsymbol{\alpha},\mathbf h}(\tilde g)\Bigr]\exp\left[n^2R_{n;\boldsymbol{\alpha}, \bf{h}}(X)\right]\\
&=
\sum_{\tilde g\in \widetilde{\mathcal W}^{(k)}_{\mathcal B}}
\exp\Bigl[n^2 U^{(\mathcal B)}_{\boldsymbol{\alpha},\mathbf h}(\tilde g)\Bigr]\sum_{X\in \mathcal{A}^{(\bf{B})}_{n}: g^{X}_{\mathcal{B}}\in \tilde{g}}\exp\left[n^2R_{n;\boldsymbol{\alpha}, \bf{h}}(X)\right] \label{eq:PF_R}
\end{align}
(the superscript $\mathcal B$ is omitted whenever this causes no ambiguity).\\
We emphasize that the inner sum in \eqref{eq:PF_R} runs over the set
$
[\,\widetilde{g}^{(\mathcal B)}\,]_n
:=
\bigl\{
X\in \mathcal{A}^{({\bf{B}}_n)}_{n}:\;
g^{X}_{\mathcal B}\in \widetilde{g}^{(\mathcal B)}
\bigr\},
$
that is, the collection of colored graphs whose checkerboard graphon,
equipped with the limiting partition $\mathcal B$,
belongs to the equivalence class
$\widetilde g^{(\mathcal B)}\in
\widetilde{\mathcal W}_{\mathcal B}^{(k)}$.
An analogous decomposition applies to the numerator of $\P^{({\bf{B}}_n)}_{n;\boldsymbol{\alpha}, \bf{h}}(X)$. Combining these observations, the Gibbs measure can be written as
\begin{equation}\label{eq:measP}
\P^{({\bf{B}}_n)}_{n;\boldsymbol{\alpha},\bf{h}}(X)=
\frac{\exp\Bigl[n^2 U^{(\mathcal B)}_{\boldsymbol{\alpha},\mathbf h}(\tilde g)\Bigr]\exp\left[n^2R_{n;\boldsymbol{\alpha}, \mathbf h}(X)\right]\mathds{1}(g^{X}_{\mathcal{B}}\in \tilde{g})}{\sum_{\tilde g\in \widetilde{\mathcal W}^{(k)}_{\mathcal B}}
\exp\Bigl[n^2 U^{(\mathcal B)}_{\boldsymbol{\alpha},\mathbf h}(\tilde g)\Bigr]\sum_{X\in \mathcal{A}^{(\bf{B})}_{n}}\exp\left[n^2R_{n;\boldsymbol{\alpha}, \mathbf h}(X)\right]\mathds{1}(g^{X}_{\mathcal{B}}\in \tilde{g})},
\end{equation}
where $X\in\mathcal{A}^{({\bf B} )}_n$.  
We now bound the measure in \eqref{eq:measP} by using \eqref{ineq:boundR}, therefore we get
\begin{equation}\label{eq:final_meas}
\begin{aligned}
&\exp{[-2\gamma\eta_{n}({\bf{B}}) n^2]}\frac{\exp\Bigl[n^2 U^{(\mathcal B)}_{\boldsymbol{\alpha},\mathbf h}(\tilde g)\Bigr]\mathds{1}(g^{X}_{\mathcal{B}}\in \tilde{g})}{\sum_{\tilde g\in \widetilde{\mathcal W}^{(k)}_{\mathcal B}}
\exp\Bigl[n^2 U^{(\mathcal B)}_{\boldsymbol{\alpha},\mathbf h}(\tilde g)\Bigr]\sum_{X\in \mathcal{A}^{(\bf{B})}_{n}}\mathds{1}(g^{X}_{\mathcal{B}}\in \tilde{g})} \\&\leq \P^{({\bf{B}}_n)}_{n;\boldsymbol{\alpha},\bf{h}}(X)\leq \exp{[2\gamma\eta_{n}({\bf{B}}) n^2]}\frac{\exp\Bigl[n^2 U^{(\mathcal B)}_{\boldsymbol{\alpha},\mathbf h}(\tilde g)\Bigr]\mathds{1}(g^{X}_{\mathcal{B}}\in \tilde{g})}{\sum_{\tilde g\in \widetilde{\mathcal W}^{(k)}_{\mathcal B}}
\exp\Bigl[n^2 U^{(\mathcal B)}_{\boldsymbol{\alpha},\mathbf h}(\tilde g)\Bigr]\sum_{X\in \mathcal{A}^{(\bf{B})}_{n}}\mathds{1}(g^{X}_{\mathcal{B}}\in \tilde{g})}. 
\end{aligned}
\end{equation}
We next seek a convenient expression for the inner sum $\sum_{X\in \mathcal{A}^{(\bf{B})}_{n}}\mathds{1}(g^{X}_{\mathcal{B}}\in \tilde{g}^{(\mathcal{B})})= |[\,\widetilde{g}^{(\mathcal B)}\,]_n|$, see the definition below \eqref{eq:PF_R}. To this end, let
$\mathbb P^{\mathrm{ER}}_{n;1/2}$
denote the standard \ER measure on
$\mathcal A_n$ with parameter $p=\frac{1}{2}$, which is 
uniform on $\mathcal A_n$.
It therefore induces a probability measure
$\widetilde{\mathbb P}^{\mathrm{ER}}_{n;1/2}$
on $\widetilde{\mathcal W}_{\mathcal B}^{(k)}$
through the map $X\mapsto g^X_{\mathcal B}$, namely
\begin{equation}\label{displ:change_meas}
\widetilde{\mathbb P}^{\mathrm{ER}}_{n;1/2}
\bigl(\widetilde g^{(\mathcal B)}\bigr)
=
\frac{\bigl|
\{X\in\mathcal A_n:
g^X_{\mathcal B}\in
\widetilde g^{(\mathcal B)}\}
\bigr|}
{2^{\binom n2}}.
\end{equation}
Thanks to the action of the color-blind map $g^{X}_{\mathcal{B}}$,
the same expression holds with $\mathcal A_n$ replaced by
$\mathcal A_n^{(\mathbf B)}$.
Consequently,
\begin{equation}\label{eq:innersum}
\sum_{X\in \mathcal A_n^{(\mathbf B)}}
\mathds 1\bigl(g^X_{\mathcal B}\in
\widetilde g^{(\mathcal B)}\bigr)
=
\bigl|[\,\widetilde g^{(\mathcal B)}\,]_n\bigr|
=
2^{\binom n2}
\widetilde{\mathbb P}^{\mathrm{ER}}_{n;1/2}
\bigl(\widetilde g^{(\mathcal B)}\bigr).
\end{equation}

We are now in a position to recast the Gibbs measures appearing on both sides of \eqref{eq:final_meas}
into a form that is most convenient for our analysis.
Summing the pointwise bounds in \eqref{eq:final_meas}
over all configurations
$X\in[\widetilde g^{(\mathcal B)}]_n$
amounts to taking the push-forward of
$\mathbb P^{(\mathbf B_n)}_{n;\boldsymbol\alpha,\mathbf h}$
under the map \eqref{map_pushforward}:
\[
\widetilde{\mathbb P}^{(\mathcal B)}_{n;\boldsymbol\alpha,\mathbf h}
(\widetilde g)
:=\mathbb P^{(\mathbf B_n)}_{{n;\boldsymbol\alpha,\mathbf h}}
\bigl(\{X\in \mathcal{A}^{({\bf{B}}_n)}_{n}: g^{X}_{\mathcal B}\in \widetilde g\}\bigr)=
\sum_{X\in[\widetilde g^{(\mathcal B)}]_n}
\mathbb P^{(\mathbf B_n)}_{n;\boldsymbol\alpha,\mathbf h}(X).
\]
Consequently, for every Borel set
$\widetilde A\subseteq
\widetilde{\mathcal W}^{(k)}_{\mathcal B}$ and each $n\geq 1$,
we obtain the bounds
$$
\exp{[-2\gamma\eta_{n}({\bf{B}}) n^2]}\widetilde{\mathbb{Q}}^{(\mathcal{B})}_{n;\boldsymbol{\alpha},\mathbf h}(\widetilde{A})\leq \widetilde{\mathbb P}^{(\mathcal B)}_{n;\boldsymbol\alpha,\mathbf h}(\widetilde{A})\leq \exp{[2\gamma\eta_{n}({\bf{B}}) n^2]}\widetilde{\mathbb{Q}}^{(\mathcal{B})}_{n;\boldsymbol{\alpha},\mathbf h}(\widetilde{A}), 
$$
where the reference measures
\(\widetilde{\mathbb Q}^{(\mathcal B)}_{n;\boldsymbol\alpha,\mathbf h}\) on
\(\widetilde{\mathcal W}^{(k)}_{\mathcal B}\) are defined (using \eqref{eq:innersum}) as

\begin{equation}\label{tilted_measures}
\widetilde{\mathbb{Q}}^{(\mathcal{B})}_{n;\boldsymbol{\alpha},\mathbf h}(\widetilde{A}) := \frac{\sum_{\widetilde{g} \in \widetilde{A}} \exp \left(n^2U^{(\mathcal{B})}_{\boldsymbol{\alpha}, \mathbf h}(\widetilde{g})\right) \widetilde{\mathbb P}^{\mathrm{ER}}_{n;1/2}(\widetilde{g})}{\sum_{\widetilde{g} \in \widetilde{\mathcal{W}}_{\mathcal{B}}^{(k)}} \exp\left(n^2U^{(\mathcal{B})}_{\boldsymbol{\alpha}, \mathbf h}(\widetilde{g})\right) \widetilde{\mathbb P}^{\mathrm{ER}}_{n;1/2}(\widetilde{g})}.
\end{equation}
We now establish a large deviation principle for the sequence
$\bigl(\widetilde{\mathbb Q}^{(\mathcal B)}_{n;\boldsymbol\alpha,\mathbf h}\bigr)_{n\ge1}$
on
$\bigl(\widetilde{\mathcal W}^{(k)}_{\mathcal B},
\delta^{(k)}_\square\bigr)$
with speed $n^2$.
By exponential equivalence, this immediately yields the same LDP for
$\bigl(\widetilde{\mathbb P}^{(\mathcal B)}_{n;\boldsymbol\alpha,\mathbf h}\bigr)_{n\ge1}$.
To this end, we apply \cite[Thm.~II.7.2(a)]{E}.
The assumptions of the theorem are satisfied:
the functional
$U^{(\mathcal B)}_{\boldsymbol\alpha,\mathbf h}$
is bounded and continuous on the compact metric space
$\bigl(\widetilde{\mathcal W}^{(k)}_{\mathcal B},
\delta^{(k)}_\square\bigr)$
(see Prop.~\ref{prop:continuity-cell-restricted}
and Lem.~\ref{compact}),
and the reference measures
$\bigl(\widetilde{\mathbb P}^{\mathrm{ER}}_{n;1/2}\bigr)_{n\ge1}$
satisfy an LDP on this space with good rate function
$\mathcal I^{(\mathcal B)}_{1/2}$
(see Lem.~\ref{ER_LDP}).
We can then conclude that
the sequence $\{\widetilde{\mathbb{Q}}_{n;\boldsymbol{\alpha},\bf{h}}\}_{n \geq 1}$ satisfies a large deviation principle with speed $n^2$ and rate function
\begin{equation}\label{tilted_rate_function}
\mathcal{I}^{(\mathcal{B})}_{\boldsymbol{\alpha},\bf{h}}(\widetilde{g}) := \mathcal{I}^{(\mathcal{B})}_{\frac{1}{2}}(\widetilde{g}) - U^{(\mathcal{B})}_{\boldsymbol{\alpha}, \bf{h}}(\widetilde{g}) - \inf_{\widetilde{g} \in \widetilde{\mathcal{W} }^{(k)}_{\mathcal{B}}} \left\{ \mathcal{I}^{(\mathcal{B})}_{\frac{1}{2}}(\widetilde{g}) - U^{(\mathcal{B})}_{\boldsymbol{\alpha}, \bf{h}}(\widetilde{g}) \right\},
\end{equation}
where $\mathcal{I}^{(\mathcal{B})}_{\frac{1}{2}}$ is obtained by setting $p=\frac{1}{2}$ in \eqref{ER_rate_function}.
Note that since $\mathcal{I}^{(\mathcal{B})}_{1/2}$ is lower semicontinuous (it is indeed a rate function, see Lem.~\ref{ER_LDP}), the function $\mathcal{I}^{(\mathcal{B})}_{\boldsymbol{\alpha},\bf{h}}$ is too (as a sum of lower semicontinuous functions).
Finally, we point out that \eqref{tilted_rate_function} coincides with
\eqref{eq:LDP-rate}, since the two expressions differ only by a
\(\tfrac{1}{2}\ln 2\) term, which cancels out. To conclude the proof, we need to show that
\begin{equation}\label{eq:limfe}
\lim_{n\to\infty}\frac{1}{n^2}\ln Z^{(\bf{B})}_{n;\boldsymbol{\alpha},\bf{h}} = f^{(\mathcal{B})}_{\boldsymbol{\alpha},\bf{h}},
\end{equation}
where we recall that
\begin{equation}
\begin{aligned}\label{prob_var}
f^{(\mathcal{B})}_{\boldsymbol{\alpha},\bf{h}}&=\\
&\sup_{\widetilde{g}\in\widetilde{\mathcal{W} }^{(k)}_{\mathcal{B}}}\Big\{\frac{1}{6}\sum_{i,j,\ell\in[k]}
\alpha_{ij\ell}\, t^{(\mathcal B)}_{ij\ell}(H_2,\widetilde g)
+
\frac12\sum_{i,j\in[k]}
h_{ij}\, t^{(\mathcal B)}_{ij}(H_1,\widetilde g)-\frac{1}{2}\mathcal{I}^{(\mathcal{B})}(\widetilde{g})\Big\}
\end{aligned}
\end{equation}
is the limiting free energy. The limit in \eqref{eq:limfe} follows as a direct application of Varadhan’s theorem
\cite[Thm.~3.4]{V}, since $f^{(\mathcal{B})}_{\boldsymbol{\alpha},\mathbf{h}}$
coincides with the Legendre transform of the rate function in
\eqref{tilted_rate_function}.
Alternatively, a more direct proof can be obtained by adapting the argument of
\cite[Thm.~3.1]{CD}. We defer this approach to the appendix (see Subsec.~\ref{alternative_der_fe}).
\end{proof}

\begin{remark}[Alternative representation of $f^{(\mathcal B)}_{\boldsymbol\alpha,\mathbf h}$]
By adding and subtracting
$\frac12\sum_{i,j\in[k]}\ln\frac{1}{1+e^{h_{ij}}}$
in \eqref{prob_var}, the last two terms can be rewritten as
\begin{align}\label{fr_en_repr}
&\frac{1}{2}\sum_{i,j \in [k]} h_{ij}
\, t_{ij}^{(\mathcal B)}(H_1,\widetilde g)
-\frac{1}{2}\mathcal I^{(\mathcal B)}(\widetilde g)
\notag\\
&\hspace{2cm}
=\frac{1}{2}\sum_{i,j\in[k]}
\int_{\mathcal B_i\times\mathcal B_j}
I_{p_{ij}}\bigl(g(x,y)\bigr)\,dx\,dy
-\frac{1}{2}\sum_{i,j\in[k]}
\ln\frac{1}{1+e^{h_{ij}}},
\end{align}
where
$p_{ij}:=\frac{e^{h_{ij}}}{1+e^{h_{ij}}}$,
and $I_{p_{ij}}$ denotes the relative entropy
introduced below \eqref{ER_rate_function}.

In the case $\boldsymbol\alpha=0$,
representation \eqref{fr_en_repr} shows that the rate function
\eqref{tilted_rate_function} reduces to
\begin{equation}\label{eq:altern_fe}
\frac{1}{2}\sum_{i,j\in[k]}
\int_{\mathcal B_i\times\mathcal B_j}
I_{p_{ij}}\bigl(g(x,y)\bigr)\,dx\,dy
-
\inf_{\widetilde g\in
\widetilde{\mathcal W}^{(k)}_{\mathcal B}}
\Bigl\{
\frac{1}{2}\sum_{i,j\in[k]}
\int_{\mathcal B_i\times\mathcal B_j}
I_{p_{ij}}\bigl(g(x,y)\bigr)\,dx\,dy
\Bigr\}.
\end{equation}
The additive constants
$\pm \frac12\sum_{i,j}\ln\frac{1}{1+e^{h_{ij}}}$
cancel out.
Moreover, the infimum in \eqref{eq:altern_fe}
is attained by the $\mathcal B$--block constant graphon
with values $(p_{ij})_{i,j\in[k]}$,
and since $I_p(p)=0$, the second term vanishes.
This recovers \cite[Eq.~(27)]{GP}.
\end{remark}

We conclude the section with the proof of Lem.~\ref{lem:r1-bound}.
\begin{proof}[proof of Lem.~\ref{lem:r1-bound}]
We prove \eqref{ineq:r1}, being \eqref{ineq:r2} analogous.
Since $0\le g^X\le 1$, from \eqref{eq:r1} we obtain\footnote{Set $\mathscr B := \mathcal B_i \times \mathcal B_j$,
$\mathscr B_n := \mathcal B_i^{(n)} \times \mathcal B_j^{(n)}$ and recall that $
\mathscr B \triangle \mathscr B_n
=
(\mathscr B \setminus \mathscr B_n)\,\cup\,(\mathscr B_n \setminus \mathscr B)$. To obtain \eqref{r1_bound}, we used on the domain of \eqref{eq:r1} the decomposition $\mathscr B
=
(\mathscr B \cap \mathscr B_n)\,\cup\,(\mathscr B \setminus \mathscr B_n)$ (similarly for $\mathscr B_n$), combined with triangle inequality.} 
\begin{align}\label{r1_bound}
\bigl|r^{(1)}_{ij,n}(X)\bigr|
&\le \int_{(\mathcal B_i\times\mathcal B_j)
             \,\triangle\, (\mathcal B_i^{(n)}\times \mathcal B_i^{(n)})}
        g^X(x,y)\,dx\,dy
\le \lambda\Bigl(
\mathcal B_i\times\mathcal B_j)
             \,\triangle\, (\mathcal B_i^{(n)}\times \mathcal B_i^{(n)})
  \Bigr). 
\end{align}
We now use the following geometric fact.

\begin{claim}\label{claim:geom}
For measurable sets $A,B,C,D\subset[0,1]$, one has
\begin{equation}\label{incl_sets}
(A\times C)\,\triangle\,(B\times D)
\subset (A\triangle B)\times[0,1]
      \;\cup\; [0,1]\times(C\triangle D),
\end{equation}
and consequently
\begin{equation}\label{dis_measures}
\lambda\bigl((A\times C)\,\triangle\,(B\times D)\bigr)
\le \lambda(A\triangle B) + \lambda(C\triangle D).
\end{equation}
\end{claim}
Applying Claim~\ref{claim:geom} with
\[
A=\mathcal B_i,\quad B=\mathcal{B}_i^{(n)},\quad
C=\mathcal B_j,\quad D=\mathcal{B}_j^{(n)},
\]
we deduce
\begin{align*}
\lambda\Bigl(
    (\mathcal B_i\times\mathcal B_j)
    \,\triangle\,
    (\mathcal{B}_i^{(n)}\times \mathcal{B}_j^{(n)})
  \Bigr)
&\le \lambda(\mathcal B_i\triangle \mathcal{B}_i^{(n)})
   + \lambda(\mathcal B_j\triangle \mathcal{B}_j^{(n)}).
\end{align*}
Hence
\begin{equation}\label{eq:r1-bound-by-1d}
\bigl|r^{(1)}_{ij,n}(X)\bigr|
\le \lambda(\mathcal B_i\triangle \mathcal{B}_i^{(n)})
   + \lambda(\mathcal B_j\triangle \mathcal{B}_j^{(n)}).
\end{equation}
Because both partitions are consecutive, we can write
\[
\mathcal B_i = [a_i,a_i+b_i),\qquad
\mathcal{B}_i^{(n)} = [c_i,c_i + w_i/n),
\qquad i=1,\dots,k,
\]
where
\[
a_i := \sum_{r<i} b_r,\qquad
c_i := \sum_{r<i} \frac{w_r}{n},
\]
and the quantities $b_i$ and $w_i$ were introduced in Assumption~\ref{ass:partition}.
Now, the symmetric difference of two intervals on the real line
is the union of at most two disjoint intervals.
In particular, its Lebesgue measure is bounded by the sum of the
absolute differences of the endpoints, namely
\[
\lambda(\mathcal B_i\triangle \mathcal{B}_i^{(n)})
 \leq |a_i - c_i| + |(a_i+b_i)-(c_i+w_i/n)|.
\]
From the triangle inequality
$
\bigl|(a_i+b_i)-(c_i+w_i/n)\bigr|
\le |a_i-c_i| + |b_i - w_i/n|,
$
we obtain
\begin{equation}\label{eq:symm-diff-ai-ci}
\lambda(\mathcal B_i\triangle \mathcal{B}_i^{(n)})
\le 2|a_i-c_i| + |b_i - w_i/n|.
\end{equation}
Now note that
\[
a_i - c_i
= \sum_{r<i} (b_r - w_r/n),
\]
and therefore
\[
|a_i-c_i|
\le \sum_{r<i} |b_r - w_r/n|
\overset{\eqref{def:eta}}{\le} (i-1)\,\eta_n
\le (k-1)\,\eta_n.
\]
Combining this with \eqref{eq:symm-diff-ai-ci} and the fact that
$|b_i-w_i/n|\le\eta_n$, we get
\[
\lambda(\mathcal B_i\triangle \mathcal{B}_i^{(n)})
\le 2(k-1)\,\eta_n + \eta_n
= (2k-1)\,\eta_n.
\]
Similarly,
$
\lambda(\mathcal B_j\triangle \mathcal{B}_j^{(n)})
\le (2k-1)\,\eta_n.
$
Plugging these bounds into \eqref{eq:r1-bound-by-1d} yields
\[
\bigl|r^{(1)}_{ij,n}(X)\bigr|
\le 
 (4k-2)\,\eta_n.
\]
This proves \eqref{ineq:r1} with $\gamma^{(1)}_k:= 4k-2$.
Applying Claim~\ref{claim:geom} twice, first with
\[
A=\mathcal B_i\times\mathcal B_j,\quad
B=\mathcal B_i^{(n)}\times\mathcal B_j^{(n)},\quad
C=\mathcal B_\ell,\quad
D=\mathcal B_\ell^{(n)},
\]
and then again to control
\(
(\mathcal B_i\times\mathcal B_j)\triangle
(\mathcal B_i^{(n)}\times\mathcal B_j^{(n)}),
\)
we deduce
\begin{align*}
&\lambda\Bigl(
(\mathcal B_i\times\mathcal B_j\times\mathcal B_\ell)
\,\triangle\,
(\mathcal{B}_i^{(n)}\times\mathcal{B}_j^{(n)}\times\mathcal{B}_\ell^{(n)})
\Bigr)\\
&\hspace{2cm}\le
\lambda(\mathcal B_i\triangle \mathcal{B}_i^{(n)})
+\lambda(\mathcal B_j\triangle \mathcal{B}_j^{(n)})
+\lambda(\mathcal B_\ell\triangle \mathcal{B}_\ell^{(n)}).
\end{align*}
This proves \eqref{ineq:r2} with $\gamma^{(2)}_k:= 6k-3$.
\end{proof}
We complete this section with the proof of Claim~\ref{claim:geom}.
\begin{proof}[Proof of Claim~\ref{claim:geom}]
Take $(x,y)\in (A\times C)\triangle(B\times D)$. Then $(x,y)$ belongs to
exactly one of $A\times C$ and $B\times D$.

\emph{Case 1:} $(x,y)\in A\times C$ but $(x,y)\notin B\times D$.
Then either $x\notin B$ or $y\notin D$.
If $x\notin B$, then $x\in A\triangle B$ and hence
$(x,y)\in (A\triangle B)\times[0,1]$.
If $x\in B$ but $y\notin D$, then $y\in C\triangle D$ and hence
$(x,y)\in [0,1]\times(C\triangle D)$.
So in either case,
\[
(x,y)\in (A\triangle B)\times[0,1]
      \;\cup\; [0,1]\times(C\triangle D).
\]

\emph{Case 2:} $(x,y)\in B\times D$ but $(x,y)\notin A\times C$.
The argument is symmetric: either $x\notin A$ or
$y\notin C$, implying again
$(x,y)\in (A\triangle B)\times[0,1]
      \cup [0,1]\times(C\triangle D)$.
Thus, the inclusion in \eqref{incl_sets} holds. Taking Lebesgue measure we obtain the bound in \eqref{dis_measures}. 
\end{proof}
\subsection{Euler-Lagrange equations}

We consider the variational problem \eqref{varpr}:
\begin{equation}\label{eq:variational-alphaijk}
\sup_{\tilde g\in\widetilde{\mathcal W}^{(k)}_{\mathcal B}}
\left\{
\frac{1}{6}\sum_{i,j,\ell=1}^k \alpha_{ij\ell}\,t^{(\mathcal B)}_{ij\ell}(H_2,\tilde g)
+\frac12\sum_{i,j=1}^k h_{ij}\, t^{(\mathcal B)}_{ij}(H_1,\tilde g)
-\frac12\,\mathcal I^{(\mathcal B)}(\tilde g)
\right\}.
\end{equation}
We recall that $\mathcal B=(\mathcal B_1,\dots,\mathcal B_k)$ is our fixed (limiting) partition of $[0,1]$ and that we set
\begin{equation}\label{setting_part}
b_i:=\lambda(\mathcal B_i),\qquad i\in[k],\qquad \sum_{i=1}^k b_i=1.
\end{equation}

\begin{proof}[Proof of Thm.~\ref{thm:EL-alphaijk}]
As a preliminary step, we rewrite the functional \eqref{eq:variational-alphaijk} in a more convenient form, in which the dependence on the limiting partition is entirely encoded in the two functions $h_{\mathcal{B}}$ and $\Delta_{\alpha;\mathcal{B}}$, defined in \eqref{eq:hB} and \eqref{eq:aalpha-kernel}, respectively. This reformulation leads to the variational problem \eqref{eq:var_fbis} below. First note that
\begin{align}
\sum_{i,j=1}^k h_{ij}\,t^{(\mathcal B)}_{ij}(H_1,g)
&=\sum_{i,j=1}^k h_{ij}\int_{\mathcal B_i\times\mathcal B_j} g(x,y)\,dx\,dy\\
&=\int_{[0,1]^2} h_{\mathcal B}(x,y)\,g(x,y)\,dx\,dy,
\label{eq:edge-integral-rewrite}
\end{align}
where we recall $
h_{\mathcal B}(x,y)=\sum_{i,j=1}^k h_{ij}\,\mathds 1_{\mathcal B_i}(x)\,\mathds 1_{\mathcal B_j}(y).
$
Similarly,
\begin{align}
\sum_{i,j,\ell=1}^k \alpha_{ij\ell}\,t^{(\mathcal B)}_{ij\ell}(H_2,g)
&=\sum_{i,j,\ell=1}^k \alpha_{ij\ell}
\int_{\mathcal B_i\times\mathcal B_j\times\mathcal B_\ell}
g(x,y)g(y,z)g(z,x)\,dx\,dy\,dz \notag\\
&=\int_{[0,1]^3} \Delta_{\boldsymbol\alpha;\mathcal{B}}(x,y,z)\,g(x,y)g(y,z)g(z,x)\,dx\,dy\,dz,
\label{eq:tri-integral-rewrite}
\end{align}
where 
$\Delta_{\boldsymbol\alpha; \mathcal{B}}(x,y,z)
=
\sum_{i,j,\ell=1}^k \alpha_{ij\ell}\,
\mathds 1_{\mathcal B_i}(x)\,\mathds 1_{\mathcal B_j}(y)\,\mathds 1_{\mathcal B_\ell}(z).
$
Hence, the objective function in \eqref{eq:variational-alphaijk} can be written as
\begin{align}\label{eq:FBalpha-integral}
\mathcal F^{(\mathcal{B})}_{\boldsymbol\alpha,\mathbf h}(g)
&=
\frac16\int_{[0,1]^3} \Delta_{\boldsymbol\alpha; \mathcal{B}}(x,y,z)\,g(x,y)g(y,z)g(z,x)\,dx\,dy\,dz\\
&+\frac12\int_{[0,1]^2} h_{\mathcal B}(x,y)\,g(x,y)\,dx\,dy
-\frac12\int_{[0,1]^2} I(g(x,y))\,dx\,dy.
\end{align}
With this representation at hand, we focus on
\begin{equation}\label{eq:var_fbis}
\sup_{\tilde{g}\in \widetilde{W}} \mathcal{F}^{(\mathcal{B})}_{\boldsymbol{\alpha},\bf h}(\tilde{g}),
\end{equation}
and we assume that $\tilde{g}^{\star}$ is a maximizer of \eqref{eq:var_fbis}. We split the proof into two steps, following the argument of \cite[Thm.~6.3]{CD}.
\begin{enumerate}
\item We start from the proof of Item~\ref{item:el:eq}.
Assume first that a representative element $g\in \tilde{g}^{\star}$ satisfies $\delta\le g\le 1-\delta$ a.e.\ for some $\delta>0$.
Let $\varphi:[0,1]^2\to\mathbb R$ be bounded, measurable and symmetric, and set $g_u(x,y):=g(x,y)+u\varphi(x,y)$, $u\in \R$. Then, $g_u$ is as well a bounded, symmetric function from $[0,1]^2$ to $\R$.
Throughout the proof, we consider $|u|$ sufficiently small, so that $g_u\in\mathcal W$.
Then, by maximality we have $\mathcal F^{(\mathcal{B})}_{\boldsymbol\alpha,\mathbf h}(g_u)\le \mathcal F^{(\mathcal{B})}_{\boldsymbol\alpha,\mathbf h}(g)$, hence
\begin{equation}\label{eq:stationarity-alphaijk}
\left.\frac{d}{du}\mathcal F^{(\mathcal{B})}_{\boldsymbol\alpha,\mathbf h}(g_u)\right|_{u=0}=0.
\end{equation}
We compute the derivative in \eqref{eq:stationarity-alphaijk}, starting from the entropy term.
Since $I$ is $C^1$ on $(0,1)$, $u$ is assumed to be sufficiently small, and $g$ stays in $[\delta,1-\delta]$, Leibniz integral rule yields
\begin{align}\label{eq:el_derivedge}
&-\frac{1}{2}\left.\frac{d}{du}\right|_{u=0}\int_{[0,1]^2} I(g_u(x,y))\,dx\,dy
\notag\\
&\hspace{4cm}=-\frac{1}{2}\int_{[0,1]^2} I'(g(x,y))\,\varphi(x,y)\,dx\,dy \notag\\
&\hspace{4cm}=-\frac12\int_{[0,1]^2} \varphi(x,y)\ln\frac{g(x,y)}{1-g(x,y)}\,dx\,dy.
\end{align}
For the edge term we have
\begin{equation}\label{eq:var-edge-alphaijk}
\left.\frac{d}{du}\right|_{u=0}\left(\frac12\int_{[0,1]^2} h_{\mathcal B}(x,y)\,g_u(x,y)\,dx\,dy\right)
=\frac12\int_{[0,1]^2} h_{\mathcal B}(x,y)\,\varphi(x,y)\,dx\,dy.
\end{equation}
Finally, for the triangle term in \eqref{eq:FBalpha-integral}, we 
define 
\[
\Xi_{\boldsymbol\alpha;\mathcal{B}}(g):=\int_{[0,1]^3} \Delta_{\boldsymbol\alpha;\mathcal{B}}(x,y,z)\,g(x,y)g(y,z)g(z,x)\,dx\,dy\,dz.
\]
For each $(x,y,z)\in[0,1]^3$ we have
\begin{align}\label{eq:product}
g_u(x,y)\,g_u(y,z)\,g_u(z,x)
&=\bigl(g(x,y)+u\varphi(x,y)\bigr)\bigl(g(y,z)+u\varphi(y,z)\bigr)\notag\\
&\hspace{3cm}\times\bigl(g(z,x)+u\varphi(z,x)\bigr)
\end{align}
(again, since 
$u$ is taken sufficiently small and all the functions involved are bounded, differentiation can be interchanged with integration).
Expanding the product in \eqref{eq:product}, and isolating the terms of order $u$ and $u^2$, we get
\begin{align*}
g_u(x,y)\,g_u(y,z)\,g_u(z,x)
&=g(x,y)\,g(y,z)\,g(z,x)\\
& +u\Big(
\varphi(x,y)\,g(y,z)\,g(z,x)
+g(x,y)\,\varphi(y,z)\,g(z,x)\\
&+g(x,y)\,g(y,z)\,\varphi(z,x)
\Big)+u^2 R_u(x,y,z),
\end{align*}
where $R_u(x,y,z)$ denotes a remainder term collecting all contributions of order at least two in $u$.
Multiplying by $\Delta_{\boldsymbol\alpha;\mathcal{B}}(x,y,z)$ and integrating, we obtain
\begin{align}\label{eq:derivative_eltr}
\left.\frac{d}{du}\right|_{u=0}\Xi_{\boldsymbol\alpha;\mathcal{B}}(g_u)
&=\int_{[0,1]^3} \Delta_{\boldsymbol\alpha;\mathcal{B}}(x,y,z)\Big(
\varphi(x,y)\,g(y,z)\,g(z,x)
\notag\\
&\hspace{-0.5cm}+g(x,y)\,\varphi(y,z)\,g(z,x)
+g(x,y)\,g(y,z)\,\varphi(z,x)
\Big)\,dx\,dy\,dz.
\end{align}
The right-hand side of \eqref{eq:derivative_eltr} is the sum of three triple integrals, which can be written in the same form, using the symmetry of
$\Delta_{\boldsymbol\alpha;\mathcal B}$ and
$g$:
\begin{align}
&\left.\frac{d}{du}\right|_{u=0}\Xi_{\boldsymbol\alpha;\mathcal{B}}(g_u)=\\
&\hspace{2.5cm}3\int_{[0,1]^2} \varphi(x,y)\left(\int_0^1 \Delta_{\boldsymbol\alpha;\mathcal{B}}(x,y,z)\,g(x,z)g(y,z)\,dz\right)\,dx\,dy.
\label{eq:psi-variation}
\end{align}
Combining \eqref{eq:stationarity-alphaijk}, \eqref{eq:el_derivedge}, \eqref{eq:var-edge-alphaijk},
and \eqref{eq:psi-variation}, we find that for every bounded symmetric $\varphi$,
\begin{align*}
\left.\frac{d}{du}\mathcal F^{(\mathcal{B})}_{\boldsymbol\alpha,\mathbf h}(g_u)\right|_{u=0}&=\int_{[0,1]^2} \varphi(x,y)\Bigg[
\frac{1}{2}\int_0^1 \Delta_{\boldsymbol\alpha;\mathcal{B}}(x,y,z)\,g(x,z)g(y,z)\,dz\\
&\hspace{2cm}+\frac12 h_{\mathcal B}(x,y)
-\frac12\ln\frac{g(x,y)}{1-g(x,y)}
\Bigg]dx\,dy=0.
\end{align*}
Choosing $\varphi(x,y)$ equal to the bracketed term (which is bounded, as $g$ is bounded away from $0$ and $1$), yields that the bracket is $0$ a.e., i.e.
\begin{equation}\label{eq:EL-pointwise-alphaijk-raw}
\ln\frac{g(x,y)}{1-g(x,y)}
=
h_{\mathcal B}(x,y)
+\int_0^1 \Delta_{\boldsymbol\alpha;\mathcal{B}}(x,y,z)\,g(x,z)g(y,z)\,dz.
\end{equation}
This proves \eqref{eq:logistic-pointwise-alphaijk}, by recalling \eqref{eq:Aalpha-def}.
\item We now move to the proof of Item~\ref{item:el:interior}.
Fix $p\in(0,1)$ and define the truncation perturbation
\[
g_{p,u}(x,y):=(1-u) g(x,y)+u\max\{g(x,y),p\},\qquad u\in[0,1].
\]
Then $g_{p,u}\in\mathcal W$ and 
\begin{equation}\label{eq:deriv_gu}
\left.\frac{d}{du}g_{p,u}(x,y)\right|_{u=0}=(p- g(x,y))_+.
\end{equation}
Set
\[
\varphi_p(x,y):=\max\{g(x,y),p\}- g(x,y)=(p-g(x,y))_+,
\]
where $(\cdot)_+$ denotes the positive part.
Then, for every $u\in[0,1]$ and a.e.\ $(x,y)\in[0,1]^2$,
\begin{equation}\label{eq:gp-affine}
g_{p,u}(x,y)= g(x,y)+u\,\varphi_p(x,y).
\end{equation}

Note that $\varphi_p$ is measurable, symmetric, and bounded with $0\le \varphi_p(x,y)\le p$.
The same steps performed in Item~\ref{item:el:eq} for computing the directional derivative at $u=0$, combined with \eqref{eq:deriv_gu}, yield
\begin{align}\label{eq:derivative-at-0}
&\left.\frac{d}{du}\mathcal F^{(\mathcal{B})}_{\boldsymbol\alpha,\mathbf h}(g_{p,u})\right|_{u=0}
=
\int_{[0,1]^2} 
\Bigg[
\frac{1}{2}\int_0^1 \Delta_{\boldsymbol\alpha;\mathcal{B}}(x,y,z)\, g(x,z) g(y,z)\,dz
\\&\hspace{2cm}+ \frac12 h_{\mathcal B}(x,y)
-\frac12\ln\frac{g(x,y)}{1- g(x,y)}
\Bigg](p- g(x,y))_+\,dx\,dy.
\end{align}
We adopt the convention that the integrand equals $+\infty$ when
$g(x,y)=0$, and equals $0$ when $g(x,y)=1$.
Now, we observe that  
\begin{equation}\label{eq:Gamma-bounded}
\|\Delta_{\boldsymbol\alpha;\mathcal{B}}\|_\infty
\le \max_{i,j,\ell\in[k]}|\alpha_{ij\ell}|
=:\alpha_{\infty}<\infty,
\end{equation}
and  
$0\le g(x,z)g(y,z)\le 1$ a.e. Hence, for a.e.\ $(x,y)$,
\begin{equation}\label{eq:triangle-lower}
\int_0^1 \Delta_{\boldsymbol\alpha;\mathcal{B}}(x,y,z)\, g(x,z) g(y,z)\,dz
\ge
-\int_0^1 |\Delta_{\boldsymbol\alpha;\mathcal{B}}(x,y,z)|\,dz
\ge -\alpha_{\infty}.
\end{equation}
Also, $h_{\mathcal B}$ is block-constant with values $h_{ij}$, so it is bounded:
\begin{equation}\label{eq:hB-bounded}
\|h_{\mathcal B}\|_\infty=\max_{i,j\in[k]}|h_{ij}|=:h_{\infty}<\infty,
\end{equation}

Combining \eqref{eq:triangle-lower} and \eqref{eq:hB-bounded}, we get the following pointwise lower bound:
for a.e.\ $(x,y)$,
\begin{equation}\label{eq:nonentropic-lower}
\frac{1}{2}\int_0^1 \Delta_{\boldsymbol\alpha;\mathcal{B}}(x,y,z) g(x,z) g(y,z)\,dz
+\frac12 h_{\mathcal B}(x,y)
\ge
-\frac{1}{2}\left(\alpha_{\infty}+ h_{\infty}\right)
=: -\kappa,
\end{equation}
where $\kappa\equiv \kappa(\boldsymbol\alpha,\mathbf h)>0$.
Now, assume by contradiction that $\lambda(\{ g<p\})>0$, where $\lambda$ denotes the Lebesgue measure.
On the set $A_p:=\{(x,y)\in[0,1]^2: g(x,y)<p\}$ we have $(p-g(x,y))_+>0$.
Moreover, for $(x,y)\in A_p$,
\[
\ln\frac{ g(x,y)}{1- g(x,y)} \le \ln\frac{p}{1-p},
\]
because the map $u\mapsto \ln\frac{u}{1-u}$ is strictly increasing on $(0,1)$.
Therefore, using \eqref{eq:nonentropic-lower}, for a.e.\ $(x,y)\in A_p$ the square bracket in \eqref{eq:derivative-at-0} satisfies
\begin{align}
&\frac{1}{2}\int_0^1 \Delta_{\boldsymbol\alpha;\mathcal B}(x,y,z)\, g(x,z) g(y,z)\,dz
+\frac12\,h_{\mathcal B}(x,y)
-\frac12\,\ln\frac{ g(x,y)}{1- g(x,y)}
\notag\\
&\ge
-\kappa-\frac12\ln\frac{p}{1-p}.
\label{eq:bracket-lower-on-Ap}
\end{align}
Now choose $p\in(0,1)$ so small that
$-\kappa-\frac12\ln\frac{p}{1-p}>0$.
Then, \eqref{eq:bracket-lower-on-Ap} implies that the bracket is strictly positive on $A_p$,
and hence the integrand in \eqref{eq:derivative-at-0} is strictly positive on $A_p$.
Since $\lambda(A_p)>0$, it follows that
\begin{equation}\label{deriv_u_positiv}
\left.\frac{d}{du}\right|_{u=0}\mathcal F^{(\mathcal{B})}_{\boldsymbol\alpha,\mathbf h}(g_{p,u})>0.
\end{equation}
Recall that $\tilde g^{\star}$ was assumed to be a maximizer of \eqref{eq:var_fbis}. 
By construction, $g_{p,u}\in\mathcal {W}$ for all $u\in[0,1]$, and satisfies 
$g_{p,0}=g$ for some representative $g\in \tilde g^{\star}$. 
It follows that the map
\[
u \longmapsto \mathcal F^{(\mathcal B)}_{\boldsymbol\alpha,\mathbf h}(g_{p,u})
\]
attains its maximum at $u=0$.
In particular its right-derivative at $0$ must satisfy
\[
\left.\frac{d}{du}\right|_{u=0^+}\mathcal F^{(\mathcal{B})}_{\boldsymbol\alpha,\mathbf h}(g_{p,u})\le 0,
\]
which contradicts \eqref{deriv_u_positiv}. Hence $\lambda(\{ g<p\})=0$, i.e.\ $ g\ge p$ a.e. A similar argument shows the upper bound in \eqref{eq:interior-alphaijk}, thus completing the proof.
\end{enumerate}

\end{proof}

\subsection{Scalar problem for the free energy}

\begin{proof}[Proof of Thm.~\ref{thm:scalar_probl}, Items~\ref{scal:item1}\& \ref{scal:item2}]
We use the non-linear part (i.e. the triangle term) of the functional in \eqref{varpr} to identify the structural
form of the maximizer. Recall the cell-restricted subgraph density  \eqref{def_graphon_hom_density}. The following H\"older inequality applies:
\begin{align*}
t^{(\bf{\mathcal{B})}}_{\ell_1 \ell_2\cdots \ell_m}(H,g) &= \int_{[\mathcal{B}_{\ell_1} \times \mathcal{B}_{\ell_2} \times\cdots \times \mathcal{B}_{\ell_{m}}]} \prod_{\{i,j\} \in \mathcal{E}(H)} g(x_i,x_j) \, dx_1 \dots dx_m . \\
&\leq \prod_{\{i,j\} \in \mathcal{E}(H)}\left[\int_{[\mathcal{B}_i \times \mathcal{B}_j \times\cdots \times \mathcal{B}_{\ell}]}  g^{E(H)}(x_i,x_j) \, dx_1 \dots dx_m\right]^{\frac{1}{|\mathcal E(H)|}}.
\end{align*}
We now specialize to the case where $H$ is a triangle (i.e., in our notation, $H\equiv H_2$); to ease notation we relabeled the indices $\ell_1,\ell_2,\ell_3$ as $i,j,\ell$. 
We get:
\begin{align}\label{ineq:Holder}
 t^{(\bf{\mathcal{B})}}_{i j \ell}(H_2,g)
&\le
\left(b_i\!\int_{\mathcal{B}_{j}\times \mathcal{B}_{\ell}} g^3(x,y) dx dy\right)^{1/3}
\left(b_j\!\int_{B_{\ell}\times B_{i}} g^3(x,y) dx dy\right)^{1/3} \notag\\
&\hspace{5cm}\times\left(b_{\ell}\!\int_{\mathcal{B}_{i}\times \mathcal{B}_{j}} g^3(x,y) dx dy\right)^{1/3} \notag\\
&=\prod_{\pi \in \mathrm{Cyc}(ij\ell)}
\left(
\,b_{\pi(i)}\,
\int_{\mathcal{B}_{\pi(j)} \times \mathcal{B}_{\pi(\ell)}}
g(x,y)^3\,dx\,dy
\right)^{\!1/3},
\end{align}
where $\mathrm{Cyc}(ij\ell)$ denotes the set of cyclic permutations of
$(i,j,\ell)$.
The above display, together with the non-negativity of $\boldsymbol{\alpha}$, yields the following inequality:
\begin{equation}\label{eq:holder_final}
\sum_{i, j, \ell=1}^{k}\alpha_{ij\ell}\;t^{(\bf{\mathcal{B})}}_{i j \ell}(H_2,g)\leq 
\sum_{i, j, \ell=1}^{k}\alpha_{ij\ell}\prod_{\pi \in \mathrm{Cyc}(ij\ell)}
\left(
\,b_{\pi(i)}\,
\int_{\mathcal{B}_{\pi(j)} \times \mathcal{B}_{\pi(\ell)}}
g(x,y)^3\,dx\,dy
\right)^{\!1/3}.
\end{equation}
Instead, for the edge term $t^{(\bf{\mathcal{B})}}_{i j}(H_1,g)$ equality is reached in \eqref{ineq:Holder}.
Then, the argument runs as follows: 
the objective functional in \eqref{varpr}, and consequently its supremum, 
is always less than or equal to the one obtained by replacing the triangle term 
with the right–hand side of \eqref{eq:holder_final} (notice that the edge term remains unchanged 
under H\"older inequality). 
If we can exhibit a function $\widetilde{g}^{\star(\mathcal{B})}$ for which the two suprema coincide, 
then the conclusion follows. Indeed, by construction, the supremum of the first 
functional is bounded above by that of the second; on the other hand, equality 
at $\widetilde{g}^{\star(\mathcal{B})}$ implies that the first supremum is also bounded below by the second. 
Combining the two inequalities, the two variational problems must therefore coincide.
Now, equality in \eqref{eq:holder_final} holds if and only if
\eqref{ineq:Holder} is an equality for each triple $(i,j,\ell)$, as the difference between the left and right-hand 
side of \eqref{ineq:Holder} is nonpositive.
By the characterization of the equality case in H\"older inequality
(see, e.g., \cite[Thm.~6.2]{F}), this is equivalent to the existence of a
measurable function $W(x,y,z)\ge 0$ and constants
$a_{ij\ell}, b_{ij\ell}, c_{ij\ell}>0$ such that, 
for every $i,j,\ell\in[k]$, the identities
\begin{align}
g(x,y)^3 &= a_{ij\ell}\,W(x,y,z), \label{eq:H1}\\
g(y,z)^3 &= b_{ij\ell}\,W(x,y,z), \label{eq:H2}\\
g(z,x)^3 &= c_{ij\ell}\,W(x,y,z), \label{eq:H3}
\end{align}
hold for a.e. $(x,y,z) \in 
\mathcal B_i\times\mathcal B_j\times\mathcal B_\ell$.
Now fix \(x\in\mathcal B_i\), \(y\in\mathcal B_j\) and vary 
\(z\in\mathcal B_\ell\). From \eqref{eq:H1} we get $\frac{g(x,y)^3 }{g(y,z)^3}= \frac{a_{ij\ell}}{b_{ij\ell}}=:\mu$ and similarly for the other ratios. Therefore conditions \eqref{eq:H1}--\eqref{eq:H3} translate into the fact that equality in \eqref{eq:holder_final} is reached if and only if, for some constants $\lambda, \mu\in \mathbb{R}^{+}$ 
\begin{equation}\label{eq:Holder_simple}
g(x,y)^3= \mu g(y,z)^3=\lambda g(z,x)^3
\qquad \text{a.e. on } \mathcal{B}_i\times \mathcal{B}_j \times \mathcal{B}_l.
\end{equation}
Fix $y\in\mathcal B_j$. 
From \eqref{eq:Holder_simple} we obtain
\[
g^{(y)}(x)^3=\mu\, g^{(y)}(z)^3
\quad\text{for a.e.\ }(x,z)\in\mathcal B_i\times\mathcal B_\ell.
\]
By Fubini's theorem, for a.e.\ $z\in\mathcal B_\ell$ the above relation 
holds for a.e.\ $x\in\mathcal B_i$.
Fix such a $z$. Then for a.e.\ $x_1,x_2\in\mathcal B_i$,
\[
g^{(y)}(x_1)^3=\mu g^{(y)}(z)^3
=
g^{(y)}(x_2)^3,
\]
so $g^{(y)}(x)$ is constant in $x$ on $\mathcal B_i$ (a.e.).
Similarly, 
we get that $g^{(y)}(z)$ is constant in $z$ on $\mathcal{B}_\ell$ (a.e.). Hence, for a.e.\ \(y\in\mathcal B_j\), there exist constants 
\(a(y), b(y)\) such that
\begin{equation}
g(x,y)=a(y)\quad\text{for a.e.\ }x\in\mathcal B_i,
\qquad
g(y,z)=b(y)\quad\text{for a.e.\ }z\in\mathcal B_\ell.
\label{eq:ayby}
\end{equation} 
We now use the second identity in \eqref{eq:Holder_simple},
\[
g(x,y)^3=\lambda g(z,x)^3
\quad\text{a.e.\ on }\mathcal B_i\times\mathcal B_j\times\mathcal B_\ell.
\]
Fix $(x,z)\in\mathcal B_i\times\mathcal B_\ell$ outside a null set 
such that the above holds for a.e.\ $y\in\mathcal B_j$.
Since the right-hand side does not depend on $y$, 
it follows that $g(x,y)$ is constant in $y$ on $\mathcal B_j$ (a.e.). 
Thus, for a.e.\ $(x,z)\in\mathcal{B}_i\times\mathcal{B}_\ell$, there exists a constant
$c(x,z)$ such that
\begin{equation}\label{eq:g_constant}
g(x,y)=c(x,z)\qquad\text{for a.e.\ }y\in\mathcal{B}_j.
\end{equation}
Now fix $x\in\mathcal{B}_i$ and pick $z_1,z_2\in\mathcal{B}_\ell$
outside a null set such that \eqref{eq:g_constant} holds for a.e.\ $y\in\mathcal{B}_j$ 
with $z=z_1$ and with $z=z_2$.
Now, we may choose $y\in\mathcal{B}_j$ such that both identities hold, and therefore
\[
c(x,z_1)=g(x,y)=c(x,z_2).
\]
Therefore $c(x,z)$ is constant in $z$ on $\mathcal{B}_\ell$ (a.e.), and we may write it as 
\begin{equation}
g(x,y)=c(x)\qquad\text{for a.e.\ }y\in\mathcal B_j.
\label{eq:cx}
\end{equation}
Let \(X_0\subset\mathcal B_i\) and \(Y_0\subset\mathcal B_j\) be full-measure
sets where \eqref{eq:ayby} and \eqref{eq:cx} hold.
For all \((x,y)\in X_0\times Y_0\),
\[
a(y)=g(x,y)=c(x).
\]
Hence for any \(y_1,y_2\in Y_0\) and \(x\in X_0\),
\[
a(y_1)=c(x)=a(y_2),
\]
so \(a(y)\) is constant on \(Y_0\).  
Since the argument is carried out for fixed blocks
$\mathcal B_i$ and $\mathcal B_j$, we denote this constant by $c_{ij}$.
It then follows that $c(x)=c_{ij}\in [0,1]$ for all $x\in X_0$.
Therefore,
\[
g(x,y)=c_{ij}
\qquad\text{for a.e.\ }(x,y)\in\mathcal B_i\times\mathcal B_j.
\]

By symmetry of the graphon we have \(c_{ij}=c_{ji}\).
This shows that any optimizer has the form 
$
g_{C}(x,y):=\sum_{1\le i,j\le k} c_{ij}\,
\mathds 1_{\mathcal B_i}(x)\,\mathds 1_{\mathcal B_j}(y).
$
Substituting this ansatz into \eqref{varpr}, we obtain the finite-dimensional
problem \eqref{eq:finite-dim-problem}, which admits at least one solution $C^{\star}$ since
the functional is continuous, and $\mathsf{C}_{\mathrm{sym}}$ is compact. 

\end{proof}

\begin{remark}
Suppose that $\alpha_{ij\ell}=1$ for all $i,j,\ell\in[k]$.
By applying again H\"older inequality to the r.h.s. of \eqref{eq:holder_final}, we get
\begin{align*}
&\sum_{i, j, \ell=1}^{k}\prod_{\pi \in \mathrm{Cyc}(ij\ell)}
\left(
\,b_{\pi(i)}\,
\int_{\mathcal{B}_{\pi(j)} \times \mathcal{B}_{\pi(\ell)}}
g(x,y)^3\,dx\,dy
\right)^{\!1/3}\\
&\leq  
\left[\sum_{i, j, \ell=1}^{k}
\,b_{i}\,
\int_{\mathcal{B}_{j} \times \mathcal{B}_{\ell}}
g(x,y)^3\,dx\,dy
\right]^{1/3}\\
&\hspace{0.5cm}\times\left[\sum_{i, j, \ell=1}^{k}
\,b_{j}\,
\int_{\mathcal{B}_{i} \times \mathcal{B}_{\ell}}
g(x,y)^3\,dx\,dy 
\right]^{1/3} \left[\sum_{i, j, \ell=1}^{k}
\,b_{\ell}\,
\int_{\mathcal{B}_{i} \times \mathcal{B}_{j}}
g(x,y)^3\,dx\,dy 
\right]^{1/3} \\
&=\int_{[0,1]^2}g^3(x,y)dx dy,
\end{align*}
where the last equality follows since, for each factor,
the sum $\sum_{r=1}^k |\mathcal{B}_r|$ can be factored out and equals $1$,
as $(\mathcal{B}_r)_{r=1}^k$ is a partition of $[0,1]$.
In particular, when $k=1$ (that is, in the single-community case) the
first and last terms in the above chain coincide.
\end{remark}

\begin{proof}[Proof of Thm.~\ref{thm:scalar_probl}, Item~ \ref{scal:item3}]
This is just a corollary of Thm.\ref{thm:EL-alphaijk}, Item~\ref{item:el:eq}. Consider $g_{C}$ as in \eqref{eq:gC-def}, i.e.
$
g_C(x,y)=\sum_{i,j=1}^k c_{ij}\,\mathds 1_{\mathcal B_i}(x)\mathds 1_{\mathcal B_j}(y).
$\\
It is easy to check that for all $i,j\in[k]$, Eq.~\eqref{eq:logistic-pointwise-alphaijk} reduces to
\begin{equation}\label{eq:fixed-point-cab-alphaijk}
c_{ij}
=
\frac{\exp\!\left(
h_{ij}+\sum_{\ell=1}^k b_\ell\,c_{i\ell}c_{j\ell}\,
\alpha_{ij\ell}
\right)}
{1+\exp\!\left(
h_{ij}+\sum_{\ell=1}^k b_\ell\,c_{i\ell}c_{j\ell}\,
\alpha_{ij\ell}
\right)}.
\end{equation}

We do this via \eqref{eq:EL-pointwise-alphaijk-raw}.
Fix $x\in\mathcal B_i$ and $y\in\mathcal B_j$.
For $z\in\mathcal B_\ell$, we have $g(x,z)=c_{i\ell}$ and $g(y,z)=c_{j\ell}$. 
Recalling from \eqref{eq:aalpha-kernel} that $\Delta_{\boldsymbol\alpha; \mathcal{B}}(x,y,z)
=
\sum_{i,j,\ell=1}^k \alpha_{ij\ell}\,
\mathds 1_{\mathcal B_i}(x)\,\mathds 1_{\mathcal B_j}(y)\,\mathds 1_{\mathcal B_\ell}(z)$, we get

\begin{align*}
\int_0^1 \Delta_{\boldsymbol\alpha;\mathcal{B}}(x,y,z)\,g(x,z)g(y,z)\,dz
&=\sum_{\ell=1}^k \int_{\mathcal B_\ell} 
\Delta_{\boldsymbol\alpha;\mathcal{B}}(x,y,z)\,c_{i\ell}c_{j\ell}\,dz\\
&=\sum_{\ell=1}^k b_\ell\,c_{i\ell}c_{j\ell}\alpha_{ij\ell}.
\end{align*}
Since $h_{\mathcal B}(x,y)=h_{ij}$ and $g(x,y)=c_{ij}$ for a.e.\ $(x,y)\in\mathcal B_i\times\mathcal B_j$,
plugging everything into \eqref{eq:EL-pointwise-alphaijk-raw} yields 
\begin{equation}\label{eq:EL-cab-alphaijk}
\ln\frac{c_{ij}}{1-c_{ij}}
=
h_{ij}+\sum_{\ell=1}^k b_\ell\,c_{i\ell}c_{j\ell}\,
\alpha_{ij\ell},
\end{equation} 
and the fixed-point equation
\eqref{eq:fixed-point-cab-alphaijk} follows.
\end{proof}

\begin{remark}[Reduction to standard edge-triangle]\label{rem:alpha-constant}
Equation \eqref{eq:fixed-point-cab-alphaijk} is a coupled fixed-point system for the matrix $C$.
It generalizes the standard fixed-point equation
$u=\frac{\exp(h+\alpha u^2)}{1+\exp(h+\alpha u^2)}$\footnote{Obtained from \eqref{eq:fixed-point-cab-alphaijk} for $k=1$; first analyzed in \cite[Lem.~12]{CDey}.}
by replacing the scalar $u^2$ with the bilinear form
$\sum_{\ell=1}^k b_\ell\,c_{i\ell}c_{j\ell}$ and by allowing block-dependent triangle weights through
$\alpha_{ij\ell}$.
In particular, the equations for different pairs $(i,j)$ are not independent.
\end{remark}

The proof Cor.~\ref{cor:unique-maximiser} is omitted from this section, as it is a straigthforward consequence of Thms.~\ref{thm:scalar_probl}--\ref{prop:small-alpha-tensor}. 

\subsection{Uniqueness}
\begin{proof}[Proof of Thm.~\ref{prop:small-alpha-tensor}]
Since $[0,1]^{k\times k}$ is a closed subset of the Banach space 
$(\mathbb{R}^{k\times k},\|\cdot\|_\infty)$, it is complete. 
The map $S^{(\mathcal B)}_{\boldsymbol\alpha;{\bf h}}$ 
is well defined on $[0,1]^{k\times k}$ and maps it into itself; 
hence Banach fixed-point theorem applies provided it is a contraction.
We now estimate its Lipschitz constant.
Let $\sigma(x)=\frac{e^x}{1+e^x}$.
Then $\sigma$ is $C^1$ and
\[
\sigma'(x)=\frac{e^x}{(1+e^x)^2}=\sigma(x)\bigl(1-\sigma(x)\bigr).
\]
Since $u(1-u)\le \frac14$ for all $u\in[0,1]$, we obtain
\[
0\le \sigma'(x)\le \frac14\qquad\text{for all }x\in\R.
\]
Hence, by the mean value theorem,
\begin{equation}\label{eq:logistic-Lip}
|\sigma(x)-\sigma(x')|\le \frac14\,|x-x'|
\qquad\text{for all }x,x'\in\R.
\end{equation}
Fix $C,\widehat{C}\in[0,1]^{k\times k}$ and $i,j\in[k]$. With a slight abuse of notation, we use the same symbol
$\mathcal T^{(\mathcal B)}_{\boldsymbol\alpha}$
to denote both the operator acting on graphons (recall \eqref{eq:Aalpha-def}) and its restriction
to $\mathcal B$--block graphons, identified with matrices
$C\in[0,1]^{k\times k}$. Define
\begin{equation}\label{eq:Tmap-tensor}
(\mathcal T^{(\mathcal B)}_{\boldsymbol\alpha}(C))_{ij}:=\sum_{\ell=1}^k b_\ell\,c_{i\ell}c_{j\ell}\alpha_{ij\ell},
\end{equation}
so that $(S^{(\mathcal B)}_{\boldsymbol\alpha; {\bf h}}(C))_{ij}=\sigma(h_{ij}+(\mathcal T^{(\mathcal B)}_{\boldsymbol\alpha}(C))_{ij})$ (recall \eqref{eq:Smap-def}).
Using \eqref{eq:logistic-Lip} with $x=h_{ij}+(S^{(\mathcal B)}_{\boldsymbol\alpha; {\bf h}}(C))_{ij}$ and $x'=h_{ij}+(S^{(\mathcal B)}_{\boldsymbol\alpha; {\bf h}}(\widehat{C}))_{ij}$, we get
\begin{equation}\label{eq:Tdiff1}
|(S^{(\mathcal B)}_{\boldsymbol\alpha; {\bf h}}(C))_{ij}-(S^{(\mathcal B)}_{\boldsymbol\alpha; {\bf h}}(\widehat{C}))_{ij}|
\le \frac14\,|(\mathcal T^{(\mathcal B)}_{\boldsymbol\alpha}(C))_{ij}-(\mathcal T^{(\mathcal B)}_{\boldsymbol\alpha}(\widehat{C}))_{ij}|.
\end{equation}
Next, by linearity and the triangle inequality,
\begin{align}
|(\mathcal T^{(\mathcal B)}_{\boldsymbol\alpha}(C))_{ij}-(\mathcal T^{(\mathcal B)}_{\boldsymbol\alpha}(\widehat{C}))_{ij}|
&=
\left|
\sum_{\ell=1}^k b_\ell\,
\alpha_{ij\ell}
\bigl(c_{i\ell}c_{j\ell}-\hat{c}_{i\ell}\hat{c}_{j\ell}\bigr)
\right| \notag\\
&\le
\sum_{\ell=1}^k b_\ell\,
|\alpha_{ij\ell}|\,
|c_{i\ell}c_{j\ell}-\hat{c}_{i\ell}\hat{c}_{j\ell}| \notag\\
&\le
\|\boldsymbol\alpha\|_\infty\sum_{\ell=1}^k b_\ell\,|c_{a\ell}c_{b\ell}-c'_{a\ell}c'_{b\ell}|,
\label{eq:Sdiff1}
\end{align}
where we recall $ \|\boldsymbol\alpha\|_\infty
=
\max_{i,j,\ell\in[k]}|\alpha_{ij\ell}|$.
We now bound the product difference. For each $\ell$,
\[
c_{i\ell}c_{j\ell}-\hat{c}_{i\ell}\hat{c}_{j\ell}
=(c_{i\ell}-\hat{c}_{i\ell})c_{j\ell}+\hat{c}_{i\ell}(c_{j\ell}-\hat{c}_{j\ell}),
\]
hence
\begin{equation}\label{eq:prod-diff}
|c_{i\ell}c_{j\ell}-\hat{c}_{i\ell}\hat{c}_{j\ell}|
\le |c_{i\ell}-\hat{c}_{i\ell}|\,c_{j\ell}+\hat{c}_{i\ell}\,|c_{j\ell}-\hat{c}_{j\ell}|.
\end{equation}
Since $C,\widehat{C}\in[0,1]^{k\times k}$, we have $c_{j\ell}\le 1$ and $\hat{c}_{i\ell}\le 1$, while
$|c_{i\ell}-\hat{c}_{i\ell}|\le \|C-\widehat{C}\|_\infty$ and $|c_{j\ell}-\hat{c}_{j\ell}|\le \|C-\widehat{C}\|_\infty$ %
Plugging these bounds into \eqref{eq:prod-diff} yields
\begin{equation}\label{eq:prod-diff2}
|c_{i\ell}c_{j\ell}-\hat{c}_{i\ell}\hat{c}_{j\ell}|
\le 2\,\|C-C'\|_\infty.
\end{equation}
Combining \eqref{eq:Sdiff1} and \eqref{eq:prod-diff2} (and recalling \eqref{setting_part}) gives
\begin{equation}\label{eq:Sdiff2}
|(\mathcal T^{(\mathcal B)}_{\boldsymbol\alpha}(C))_{ij}-(\mathcal T^{(\mathcal B)}_{\boldsymbol\alpha}(\widehat{C}))_{ij}|
\leq
2\|\boldsymbol\alpha\|_\infty\,\|C-\widehat{C}\|_\infty.
\end{equation}
Finally, injecting \eqref{eq:Sdiff2} into \eqref{eq:Tdiff1} yields, for all $i,j$,
\[
|(S^{(\mathcal B)}_{\boldsymbol\alpha; {\bf h}}(C))_{ij}-(S^{(\mathcal B)}_{\boldsymbol\alpha; {\bf h}}(\widehat{C}))_{ij}|
\leq 
\frac{1}{2}\,\|\boldsymbol\alpha\|_\infty\,\|C-\widehat{C}\|_\infty.
\]
Taking the maximum over $(i,j)$ proves the Lipschitz bound with constant
$L= \frac12 \|\boldsymbol{\alpha}\|_\infty$, yielding \eqref{eq:Lipschitz-constant-tensor}.
If $L<1$, then $S^{(\mathcal B)}_{\boldsymbol\alpha; {\bf h}}$ is a contraction on the complete metric space $[0,1]^{k\times k}$.
By Banach's fixed-point theorem, $S^{(\mathcal B)}_{\boldsymbol\alpha; {\bf h}}$ admits a unique fixed point $C^{\star}\equiv C^{\star}(\boldsymbol{\alpha},{\bf h})\in[0,1]^{k\times k}$.
Finally, by definition  \eqref{eq:Smap-def}, the fixed-point equation $S^{(\mathcal B)}_{\boldsymbol\alpha; {\bf h}}(C)=C$
is exactly the componentwise system \eqref{eq:fixed-point-scalar}.
Therefore the fixed point $C^{\star}$ is the unique solution of \eqref{eq:fixed-point-scalar}. 
\end{proof}

\subsection{SLLN}

\begin{proof}[Proof of Theorem~\ref{thm_ExpC}]
Define the finite-size free energy
\[
f^{{(\bf B})}_{n;\boldsymbol\alpha,\mathbf h}:=\frac{1}{n^2}\ln Z^{({\bf B})}_{n;\boldsymbol\alpha,\mathbf h},
\]
where $Z^{({\bf{ B}})}_{n;\boldsymbol\alpha,\mathbf h}$ is the partition function associated with
the Hamiltonian \eqref{Hamilt_ERG2}. 
For $s\in\mathbb{R}$, we introduce the scaled cumulant generating function associated with
the edge count $E_n$, defined as
\begin{equation}\label{eq:scgf-def}
c^{({\bf B})}_n(s)
:=
\frac{2}{n^2}\ln
\mathbb{E}^{({\bf{B}})}_{n;\boldsymbol\alpha,\mathbf h}
\bigl[\exp\bigl(s\,E_n\bigr)\bigr],
\end{equation}
where $\mathbb{E}^{({\bf{B}})}_{n;\boldsymbol\alpha,\mathbf h}$ denotes expectation with respect
to the Gibbs measure induced by the Hamiltonian
$\mathcal H^{(\mathbf B)}_{n;\boldsymbol\alpha,\mathbf h}$.
We prove exponential convergence of the edge density via the sequence of functions \eqref{eq:scgf-def}. By \cite[II.6.3]{E}, this boils down to proving that the limiting cumulant generating 
function $
c^{({\mathcal B})}(s)$ (see \eqref{eq:c-limit} below) is differentiable at $s=0$ and that its derivative at the origin 
coincides with the claimed limit.
A direct calculation yields
\begin{align}
c^{({\bf B})}_n(s)
&=
\frac{2}{n^2}\ln
\sum_{X\in\mathcal A^{({\bf{B}})}_n}
\frac{
\exp\!\bigl(
\mathcal H^{(\mathbf B)}_{n;\boldsymbol\alpha,\mathbf h}(X)
+s\,E_n(X)
\bigr)
}{
Z^{(\mathbf B)}_{n;\boldsymbol\alpha,\mathbf h}
}.
\label{eq:cgf-start}
\end{align}
Since $
E_n(X)
=
\sum_{i,j\in[k]}
\ \sum_{\substack{u\in B_i^{(n)},\ v\in B_j^{(n)}\\ u<v}}
X_{uv}
$ we rewrite the numerator in \eqref{eq:cgf-start} as
\begin{equation}\label{eq:identity_s}
\mathcal H^{(\mathbf B)}_{n;\boldsymbol\alpha,\mathbf h}(X)
+s\,E_n(X)
=
\mathcal H^{(\mathbf B)}_{n;\boldsymbol\alpha,\mathbf h+s\mathbf 1}(X),
\end{equation}
where $\mathbf 1\in\mathbb{R}^{k\times k}$ denotes the identity matrix.
This yields
\begin{equation}\label{eq:ratio-Z}
c^{({\bf B})}_n(s)
=
\frac{2}{n^2}
\ln
\frac{
Z^{(\mathbf B)}_{n;\boldsymbol\alpha,\mathbf h+s\mathbf 1}
}{
Z^{(\mathbf B)}_{n;\boldsymbol\alpha,\mathbf h}
}
=
2\bigl(
f^{{(\bf B})}_{n;\boldsymbol\alpha,\mathbf h+s\mathbf 1}
-
f^{{(\bf B})}_{n;\boldsymbol\alpha,\mathbf h}
\bigr).
\end{equation}
By Thm~\ref{thm:LDP-free-energy}, the infinite-volume free energy
\[
f^{{(\mathcal B})}_{\boldsymbol\alpha,\mathbf h}=\lim_{n\to\infty}\frac{1}{n^2}\ln Z^{{(\bf B})}_{n;\boldsymbol\alpha,\mathbf h}
\]
exists and is finite for any real ${\bf{h}}\in \R^{k\times k}$. Therefore, for every fixed $s\in\R$,
\begin{equation}\label{eq:c-limit}
c^{({\mathcal B})}(s):=\lim_{n\to\infty}c^{({\bf B})}_n(s)
=
2\bigl(f^{({\mathcal B})}_{\boldsymbol\alpha,\mathbf h+s\mathbf 1}-f^{({\mathcal B})}_{\boldsymbol\alpha,\mathbf h}\bigr)
\end{equation}
exists and is finite.
Recall $0\le \|\boldsymbol\alpha\|_\infty<2$.
Then, by Thm.~\ref{thm:scalar_probl} and Cor.~\ref{cor:unique-maximiser},
$f^{(\mathcal{B})}_{\boldsymbol\alpha,\mathbf h+s\mathbf 1}$ is characterized by a scalar problem, which admits a unique maximizer 
$\widetilde g^{\star(\mathcal B)}$ with representative  $(g_{C^{\star}},\mathcal{B}^{(k)})$, where
$C^\star\equiv C^\star(s,\boldsymbol{\alpha}, {\bf h})$ solves \eqref{eq:fixed-point-cab-alphaijk}. To ease notation, in the following we drop the dependence on $\boldsymbol{\alpha}, {\bf h}$ and we simply write  $C^\star(s)$.
We now compute the derivative of $s\mapsto f^{(\mathcal{B})}_{\boldsymbol\alpha,\mathbf h+s\mathbf 1}$
at $s=0$. Note that the limiting free energy might be also rewritten as
\begin{equation}\label{eq:ft-as-Psi}
f^{(\mathcal{B})}_{\boldsymbol\alpha,\mathbf h+s\mathbf 1}=\Psi\bigl(s,C^\star(s)\bigr),
\end{equation}
where $\Psi(s,C)$ is the finite-dimensional objective function
\[
\Psi(s,C)
:=
\frac{1}{6}\sum_{i,j,\ell=1}^k \alpha_{ij\ell}\,t^{(\mathcal B)}_{ij\ell}(H_2,g_C)
+\frac12\sum_{i,j=1}^k (h_{ij}+s)\,t^{(\mathcal B)}_{ij}(H_1,g_C)
-\frac12\,\mathcal I^{(\mathcal B)}(g_C),
\]
for $(s,C)\in\R\times \mathsf{C}_{\mathrm{sym}}$.
Now, $\mathsf C_{\mathrm{sym}}$ is compact. Moreover, one can readily verify that $\Psi$ is continuous in $(s,C)$, and the partial derivative $\partial_s\Psi$ exists and is continuous in $C$. Since the maximizer is unique by Thm.~\ref{prop:small-alpha-tensor}, Danskin’s theorem  \footnote{Alternatively, one may  differentiate $\Psi(s,C^\star(s))$ directly via the chain rule; however, this would require sufficient regularity of the optimizer $s\mapsto C^\star(s)$ (e.g., differentiability), which is what Danskin’s theorem provides under the above assumptions.}  \cite[Thms.~I and II]{D} yields the differentiability of
$
s\mapsto \sup_{C\in\mathsf C_{\mathrm{sym}}}\Psi(s,C),
$
and
\begin{equation}\label{eq:derivative-reduced}
\frac{d}{ds} f^{(\mathcal{B})}_{\boldsymbol\alpha,\mathbf h+s\mathbf 1}
=
\partial_s \Psi\bigl(s,C^\star(s)\bigr)=\frac12\sum_{i,j=1}^k t^{(\mathcal B)}_{ij}(H_1,g_{C^{\star}}).
\end{equation}
Evaluating at $s=0$ we get 
\begin{equation}
\frac{d}{ds}\Big|_{s=0} f_{\boldsymbol\alpha,\mathbf h+s\mathbf 1}
=
\frac12\sum_{i,j=1}^k b_i b_j\,c^\star_{ij},
\end{equation}
being 
$t^{(\mathcal B)}_{ij}(H_1,g_{C^\star})=b_i b_j\,c^\star_{ij}$.
Finally, from 
\eqref{eq:c-limit} we obtain
$c'^{({\mathcal B})}(0)
=
2\,\frac{d}{ds}\Big|_{s=0} f_{\boldsymbol\alpha,\mathbf h+s\mathbf 1}$,
thus showing \eqref{LLN}.
\end{proof}
\begin{proof}[Proof of Theorem~\ref{thm_LLN}]
The proof immediately follows from \eqref{exp_conv}, combined with Rem.~\ref{rem:kolmogorov-block}, as
exponential convergence implies almost sure convergence as a consequence of Borel-Cantelli lemma (see e.g. \cite[Thm. II.6.4]{E}). Hence
\[
\frac{2E_n}{n^2}
\xrightarrow[n\to\infty]{\ \mathrm{a.s.}\ }
c'^{({\mathcal B})}(0)=
\sum_{i,j=1}^k b_i b_j\,c^\star_{ij},
\]
which proves \eqref{LLN}.
\end{proof}

\appendix
\section{Other helpful lemmas} \label{sec_append}

\subsection{\ER LDP} 
Let $(\mathbb{P}_{n;p}^{\mathrm{ER}})_{n \geq 1}$ be the sequence of measures of a (dense) \ER random graph on the space $\mathcal{G}_n$.
Via the map \eqref{map_pushforward}, this sequence induces a family of measures
$(\widetilde{\mathbb{P}}_{n;p}^{\mathrm{ER}})_{n \geq 1}$ on
$\widetilde{W}^{(k)}_{\mathcal{B}}$.
Moreover, the LDP proved in \cite{CV11} carries over to this space.

\begin{lemma}[LDP for \ER in the space $\widetilde{W}^{(k)}_{\mathcal{B}}$] \label{ER_LDP}

For each fixed $p \in (0,1)$, the sequence $(\widetilde{\mathbb{P}}_{n;p}^{\mathrm{ER}})_{n \geq 1}$ satisfies a large deviation principle on the space $(\widetilde{W}^{(k)}_{\mathcal{B}},\delta^{(k)}_{\square})$, with speed $n^2$ and good rate function

\begin{equation}\label{ER_rate_function}
\mathcal{I}^{(\mathcal{B})}_p(\widetilde{g}) := \frac{1}{2} \sum_{i,j \in [k]}\int_{\mathcal{B}_{i}\times \mathcal{B}_{j}}  I_p(g(x,y)) \, dx \, dy,
\end{equation}
where $g$ denotes the first component of any representative 
$g_{\mathcal B} = (g,\mathcal B^{(k)})$ 
of the equivalence class $\widetilde g \in \widetilde{W}^{(k)}_{\mathcal{B}}$, and
$
I_p(u) := u \ln \frac{u}{p} + (1-u)\ln \frac{1-u}{1-p}$, $u \in [0,1]$.
\end{lemma}

\begin{proof}
This is an immediate consequence of the contraction principle (see, e.g. \cite[Thm. 4.2.1]{DZ}). Indeed, the sequence $(\mathbb{P}_{n;p}^{\mathrm{ER}})_{n \geq 1}$ satisfies an LDP on $\mathcal{W}$ in the weak topology, with the lower semicontinuous rate function \eqref{ER_rate_function} (see \cite[Thm.2.2]{CV11}).
Moreover, the map $g\mapsto g_{\mathcal{B}}\mapsto \tilde{g}^{(\mathcal{B})}$ across the spaces $ \mathcal{W}\to \mathcal{W}^{(k)}_{\mathcal{B}}\to \widetilde{\mathcal{W}}^{(k)}_{\mathcal{B}}$ is continuous. Indeed, the first map $g\mapsto g_\mathcal{B}$  is continuous by \eqref{ineq:comp_metrics}. The projection to the quotient space is also continuous. Indeed, by \eqref{ct2} (recall that we set $g_\mathcal{B}=(g, \mathcal{B}^{(k)})$),
\[
\delta_{\square}^{(k)}\!\left(\widetilde{(g_1,\mathcal B^{(k)})},\widetilde{(g_2,\mathcal B^{(k)})}\right)
=\inf_{\sigma\in\Sigma_{\mathcal B}}
d_{\square}^{(k)}\!\left((g_1,\mathcal B^{(k)}),(g_2,\mathcal B^{(k)})^{\sigma}\right).
\]
Since $\mathrm{id}\in\Sigma_{\mathcal B}$, we may choose $\sigma=\mathrm{id}$, obtaining
\begin{align*}
\delta_{\square}^{(k)}\!\left(\widetilde{(g_1,\mathcal B^{(k)})},\widetilde{(g_2,\mathcal B^{(k)})}\right)
&\le
d_{\square}^{(k)}\!\left((g_1,\mathcal B^{(k)}),(g_2,\mathcal B^{(k)})^{\mathrm{id}}\right)\\
&=
d_{\square}^{(k)}\!\left((g_1,\mathcal B^{(k)}),(g_2,\mathcal B^{(k)})\right).
\end{align*}
Finally, one can immediately check that the rate function obtained by contraction coincides with 
\eqref{ER_rate_function}; this follows from the fact that the partition 
$\mathcal B^{(k)}$ can be reabsorbed into the full integral over $[0,1]^2$ 
and therefore does not affect the value of the functional. 
\end{proof}

\subsection{Continuity of cell-restricted homomorphism densities }

Recall the colored cut-type metric \eqref{cutcol} defined by
\begin{align}\label{eq:recall:cold}
&d_{\square}^{(k)}((g_1, \mathcal{B}^{(k)}), (g_2, \mathcal{B}^{(k)})) \notag\\
&\hspace{2cm}:= \sup_{\mathcal C, \mathcal D \subseteq [0,1]} \sum_{i,j=1}^{k} \left| \int_{\mathcal C \times \mathcal D} \mathds{1}_{\mathcal{B}_i \times \mathcal{B}_j}(x,y) \left( g_1(x,y) - g_2(x,y) \right) \, dx\,dy \right|.
\end{align}
The following two-sided comparison is useful.
\begin{lemma}[Comparison with the cut norm]\label{lem:cut-comparison}
Let $(g_1,\mathcal{B}^{(k)}) , (g_2, \mathcal{B}^{(k)}) \in \mathcal{W}_{\mathcal{B}}$.
Then
\begin{equation}\label{ineq:comp_metrics}
\| g_1 - g_2 \|_{\square}
\;\le\;
d_{\square}^{(k)}\bigl((g_1,\mathcal B^{(k)}),(g_2,\mathcal B^{(k)})\bigr)
\;\le\;
k^2\,\| g_1 - g_2 \|_{\square},
\end{equation}
where
\begin{equation}\label{eq:recall_cutn}
\| g_1 - g_2 \|_{\square}
:=
\sup_{\mathcal C,\mathcal D\subseteq[0,1]}
\left|
\int_{\mathcal C\times\mathcal D}
\bigl(g_1(x,y)-g_2(x,y)\bigr)\,dx\,dy
\right|.
\end{equation}
\end{lemma}
\begin{proof}
We first prove the lower bound, starting from the integral in \eqref{eq:recall_cutn}.
Fix measurable sets \( \mathcal C,\mathcal D\subseteq[0,1] \).
Then
\begin{align*}
&\int_{\mathcal C\times\mathcal D}
\bigl(g_1(x,y)-g_2(x,y)\bigr)\,dx\,dy
=\\
&\hspace{2.5cm}\sum_{i,j=1}^k
\int_{\mathcal C\times\mathcal D}
\mathds{1}_{\mathcal B_i\times\mathcal B_j}(x,y)
\bigl(g_1(x,y)-g_2(x,y)\bigr)\,dx\,dy.
\end{align*}
Taking absolute values and using the triangle inequality,
\[
\left|
\int_{\mathcal C\times\mathcal D}\bigl(g_1(x,y)-g_2(x,y)\bigr)\,dx\,dy
\right|
\le
\sum_{i,j=1}^k
\left|
\int_{\mathcal C\times\mathcal D}
\mathds{1}_{\mathcal B_i\times\mathcal B_j}(x,y)
(g_1-g_2)
\right|.
\]
Taking the supremum over \( \mathcal C,\mathcal D \) yields
\(
\|g_1-g_2\|_{\square}
\le
d_{\square}^{(k)}((g_1,\mathcal B^{(k)}),(g_2,\mathcal B^{(k)}))
\) (recall \eqref{eq:recall:cold}).
For the upper bound, fix measurable sets \( \mathcal C, \mathcal D \subseteq [0,1] \). For each \( i,j \in \{1, \dots, k\} \), define \( \mathcal C_i := \mathcal C \cap \mathcal B_i \), \( \mathcal D_j := \mathcal D \cap \mathcal{B}_j \). Then
\begin{align}
&\left| \int_{\mathcal C \times \mathcal D} \mathds{1}_{\mathcal{B}_i \times \mathcal{B}_j}(x,y) (g_1(x,y) - g_2(x,y)) \, dx\,dy \right| \notag\\
&\hspace{2cm}= \left| \int_{\mathcal C_i \times \mathcal D_j} (g_1(x,y) - g_2(x,y)) \, dx\,dy \right| \leq \| g_1 - g_2 \|_{\square}.
\end{align}
Summing over all \( i,j \in \{1,\dots,k\} \), we get
\[
\sum_{i,j=1}^{k} \left| \int_{\mathcal C_i \times \mathcal D_j} (g_1(x,y) - g_2(x,y)) \, dx\,dy \right| \leq k^2  \| g_1 - g_2 \|_{\square}.
\]
Taking the supremum over all \( \mathcal C, \mathcal D \subseteq [0,1] \) yields the upper bound in \eqref{ineq:comp_metrics}.
\end{proof}

\begin{proposition}[Continuity of cell-restricted densities]\label{prop:continuity-cell-restricted}
Let $H$ be any finite simple graph on $V(H)=[m]$ with $\mathcal{E}(H)=\{e_1,\dots,e_E\}$, and $E:=|\mathcal E(H)|$ .
Fix $\boldsymbol\ell\in[k]^m$. Then, for all $(g_1,\mathcal{B}^{(k)}) , (g_2, \mathcal{B}^{(k)}) \in \mathcal{W}_{\mathcal{B}}$,

\begin{equation}\label{eq:continuity-bound}
\bigl|
t^{(\mathcal B)}_{\boldsymbol\ell}(H,g_1)
-
t^{(\mathcal B)}_{\boldsymbol\ell}(H,g_2)
\bigr|
\le
|\mathcal E(H)|\,
d_{\square}^{(k)}\bigl((g_1,\mathcal B^{(k)}),(g_2,\mathcal B^{(k)})\bigr).
\end{equation}
In particular, the map
\[
t^{(\mathcal B)}_{\boldsymbol\ell}(H,\cdot):
\bigl(\widetilde{W}_{\mathcal{B}}^{(k)},\delta^{(k)}_{\square}\bigr)\longrightarrow [0,1]
\]
is continuous with respect to the cut distance \eqref{ct2}.
\end{proposition}

\begin{proof}[Proof of Prop.~\ref{prop:continuity-cell-restricted}]

We show the proposition for the triangle subgraph and then we generalize it. 
Let $H=H_2$ on $V(H)=[3]$ with edges
\[
e_1=\{1,2\},\qquad e_2=\{2,3\},\qquad e_3=\{1,3\}.
\]
Fix $\boldsymbol\ell=(\ell_1,\ell_2,\ell_3)\in[k]^3$. 
We show that for every pair  $(g_1,\mathcal{B}^{(k)}) , (g_2, \mathcal{B}^{(k)}) \in \mathcal{W}_{\mathcal{B}}$,
\begin{equation}\label{eq:triangle-cell-Lipschitz}
\bigl|t^{(\mathcal B)}_{\boldsymbol\ell}(H_2,g_1)-t^{(\mathcal B)}_{\boldsymbol\ell}(H_2,g_2)\bigr|
\le 3\,\|g_1-g_2\|_\square
\overset{\eqref{ineq:comp_metrics}}{\le} 3\,d_\square^{(k)}\bigl((g_1,\mathcal B^{(k)}),(g_2,\mathcal B^{(k)})\bigr),
\end{equation}
where the cut-norm is defined in \eqref{eq:recall_cutn}. 
In particular, this shows that $(g,\mathcal B^{(k)})\mapsto t^{(\mathcal B)}_{\boldsymbol\ell}(H_2,g)$ is continuous
with respect to $d_\square^{(k)}$.
We now set $\mathbf x:=(x_1,x_2,x_3)\in[0,1]^3$ and $d\mathbf x:=dx_1dx_2dx_3$.
Define the cell indicator
\[
\mathds{1}_{\boldsymbol\ell}(\mathbf{x}):=
\prod_{v=1}^3 \mathbf 1_{\mathcal B_{\ell_v}}(x_v)
=
\mathds 1_{\mathcal B_{\ell_1}}(x_1)\,\mathds 1_{\mathcal B_{\ell_2}}(x_2)\,\mathds 1_{\mathcal B_{\ell_3}}(x_3).
\]
By definition of the cell-restricted triangle density,
\begin{equation}\label{eq:triangle-density-def}
t^{(\mathcal B)}_{\boldsymbol\ell}(H_2,g)
=
\int_{[0,1]^3}
\mathds 1_{\boldsymbol\ell}(\mathbf x)\,
g(x_1,x_2)\,g(x_2,x_3)\,g(x_1,x_3)\,d\mathbf x.
\end{equation}
For $\mathbf x\in[0,1]^3$ set
\begin{equation}\label{eq:P-defs}
\begin{aligned}
P_3(\mathbf x)&:=g_1(x_1,x_2)\,g_1(x_2,x_3)\,g_1(x_1,x_3),\\
P_2(\mathbf x)&:=g_1(x_1,x_2)\,g_1(x_2,x_3)\,g_2(x_1,x_3),\\
P_1(\mathbf x)&:=g_1(x_1,x_2)\,g_2(x_2,x_3)\,g_2(x_1,x_3),\\
P_0(\mathbf x)&:=g_2(x_1,x_2)\,g_2(x_2,x_3)\,g_2(x_1,x_3).
\end{aligned}
\end{equation}
Then $P_3-P_0=(P_3-P_2)+(P_2-P_1)+(P_1-P_0)$, and each increment factorizes as
\begin{equation}\label{eq:K3-inc-1}
\begin{aligned}
P_3(\mathbf x)-P_2(\mathbf x)
&=
g_1(x_1,x_2)\,g_1(x_2,x_3)\,\bigl(g_1-g_2\bigr)(x_1,x_3),
\\
P_2(\mathbf x)-P_1(\mathbf x)
&=
g_1(x_1,x_2)\,\bigl(g_1-g_2\bigr)(x_2,x_3)\,g_2(x_1,x_3),
\\
P_1(\mathbf x)-P_0(\mathbf x)
&=
\bigl(g_1-g_2\bigr)(x_1,x_2)\,g_2(x_2,x_3)\,g_2(x_1,x_3).
\end{aligned}
\end{equation}
Using \eqref{eq:triangle-density-def} with $g=g_1$ and $g=g_2$, and the
triangle inequality, we get
\begin{align}
\bigl|t^{(\mathcal B)}_{\boldsymbol\ell}(H_2,g_1)-t^{(\mathcal B)}_{\boldsymbol\ell}(H_2,g_2)\bigr|
&=
\left|
\int_{[0,1]^3}\mathds 1_{\boldsymbol\ell}(\mathbf x)\,\bigl(P_3(\mathbf x)-P_0(\mathbf x)\bigr)\,d\mathbf x
\right|
\notag\\
&\qquad\le
I_1+I_2+I_3,
\label{eq:K3-sum-errors}
\end{align}
where the three terms collect the integrals  corresponding to \eqref{eq:K3-inc-1}:
\begin{align}
I_3
&:=
\left|
\int_{[0,1]^3}
\mathds 1_{\boldsymbol\ell}(\mathbf x)\,
g_1(x_1,x_2)\,g_1(x_2,x_3)\,\bigl(g_1-g_2\bigr)(x_1,x_3)\,d\mathbf x
\right|,
\label{eq:K3-I3}\\
I_2
&:=
\left|
\int_{[0,1]^3}
\mathds 1_{\boldsymbol\ell}(\mathbf x)\,
g_1(x_1,x_2)\,\bigl(g_1-g_2\bigr)(x_2,x_3)\,g_2(x_1,x_3)\,d\mathbf x
\right|,
\label{eq:K3-I2}\\
I_1
&:=
\left|
\int_{[0,1]^3}
\mathds 1_{\boldsymbol\ell}(\mathbf x)\,
\bigl(g_1-g_2\bigr)(x_1,x_2)\,g_2(x_2,x_3)\,g_2(x_1,x_3)\,d\mathbf x
\right|.
\label{eq:K3-I1}
\end{align}

Now, the idea is to control each of the above terms by $\|g_1-g_2\|_\square$.
Without loss of generality we consider $I_3$; the bounds for $I_2$ and $I_1$ follow by the same argument,
after permuting the roles of $(x_1,x_2,x_3)$.
Fix $(x_1,x_3)\in[0,1]^2$. Define $F:=g_1-g_2$, and
\begin{equation}\label{eq:R3-def-clean}
R_3(x_1,x_3)
:=
\int_{[0,1]} \mathds 1_{\mathcal B_{\ell_2}}(x_2)\,g_1(x_1,x_2)\,g_1(x_2,x_3)\,dx_2.
\end{equation}

By Fubini's theorem applied to the bounded integrand in \eqref{eq:K3-I3}, we may integrate out $x_2$ first:
\begin{equation}\label{eq:I3-fubini-clean}
I_3
=
\left|
\int_{[0,1]^2}
F(x_1,x_3)\,
\mathds 1_{\mathcal B_{\ell_1}}(x_1)\,\mathds 1_{\mathcal B_{\ell_3}}(x_3)\,R_3(x_1,x_3)
\,dx_1\,dx_3
\right|.
\end{equation}
For each fixed $x_1$, the function
\[
v_{x_1}(x_3):=\mathds 1_{\mathcal B_{\ell_3}}(x_3)\,R_3(x_1,x_3)
\]
is measurable and takes values in $[0,1]$.
Hence, by the characterization of the cut norm as a supremum over bounded 
measurable test functions \cite[Eq.~(4.3) and Remark~4.1]{S}, 
applied to the functions
\[
u(x_1):=\mathds 1_{\mathcal B_{\ell_1}}(x_1),
\qquad
v(x_3):=v_{x_1}(x_3),
\]
we infer from \eqref{eq:I3-fubini-clean} that
\begin{equation}\label{bound_I3}
I_3 \le \|F\|_\square=\|g_1-g_2\|_\square.
\end{equation}

The same argument yields $I_2\le \|g_1-g_2\|_\square$ and $I_1\le \|g_1-g_2\|_\square$.
Therefore, by \eqref{eq:K3-sum-errors},
\[
\bigl|t^{(\mathcal B)}_{\boldsymbol\ell}(H_2,g_1)-t^{(\mathcal B)}_{\boldsymbol\ell}(H_2,g_2)\bigr|
\le I_1+I_2+I_3
\le 3\,\|g_1-g_2\|_\square.
\]
Finally, Lemma~\ref{lem:cut-comparison} implies
$\|g_1-g_2\|_\square\le d_\square^{(k)}\bigl((g_1,\mathcal B^{(k)}),(g_2,\mathcal B^{(k)})\bigr)$,
and consequently \eqref{eq:triangle-cell-Lipschitz}.
The continuity on the quotient space follows from \eqref{eq:continuity-bound}, together with the invariance of 
$t^{(\mathcal B)}_{\boldsymbol\ell}(H,\cdot)$ under $\Sigma_{\mathcal B}$.
\end{proof}

\begin{remark}[General graphs]\label{rem:general-H}
Repeating the previous steps for a general simple graph $H$ with $V(H)=[m]$ and $E$ edges, one can express the analog of \eqref{eq:K3-inc-1} as
\[
P_s(\mathbf x):=\prod_{r\le s} g_1(x_{a_r},x_{b_r})\prod_{r>s} g_2(x_{a_r},x_{b_r}),
\qquad s=0,1,\dots,E,
\]
thus obtaining the decomposition (c.f. \eqref{eq:K3-sum-errors})
\begin{equation}\label{eq:Decomp}
|t^{(\mathcal B)}_{\boldsymbol\ell}(H,g_1)-t^{(\mathcal B)}_{\boldsymbol\ell}(H,g_2)|
=\sum_{s=1}^E I_s.
\end{equation}
Each $I_s$ is the absolute value of an integral whose integrand contains exactly one factor $(g_1-g_2)(x_{a_s},x_{b_s})$
and all other factors are in $[0,1]$.
Treating $(x_{a_s},x_{b_s})$ as fixed parameters and integrating over the remaining $m-2$ variables,
one obtains, for each $s\in E$, a representation analogous to \eqref{bound_I3}:
\[
I_s
\le \|g_1-g_2\|_\square.
\]
Finally, from \eqref{eq:Decomp} we get
\[
\bigl|t^{(\mathcal B)}_{\boldsymbol\ell}(H,g_1)-t^{(\mathcal B)}_{\boldsymbol\ell}(H,g_2)\bigr|
\le |\mathcal{E}(H)|\,\|g_1-g_2\|_\square,
\]
and \eqref{ineq:comp_metrics} allows to conclude as before. 
\end{remark}

\subsection{Alternative derivation of the free energy} \label{alternative_der_fe}
\begin{proof}[Proof of Thm.~\ref{thm:LDP-free-energy} (free energy)]
We aim to show that
$$
\lim_{n\to\infty}\frac{1}{n^2}\ln Z^{(\bf{B})}_{n;\boldsymbol{\alpha},\bf{h}} = f^{(\mathcal{B})}_{\boldsymbol{\alpha},\bf{h}},
$$
where $f^{(\mathcal{B})}_{\boldsymbol{\alpha},\bf{h}}$ coincides with the r.h.s. of \eqref{varpr}.
This conclusion might be already suggested by \eqref{eq:PF_R}, as an application of Laplace’s method (see, e.g. \cite[Subsec. 3.3.2]{T}).
A rigorous and more direct argument, however, can be obtained by adapting the proof in \cite[Thm. 3.1]{CD}.
Indeed, the two key ingredients of the strategy, namely the continuity of subgraph densities
and the large deviation principle for the sequence of Erd\H{o}s-R\'enyi measures, %
remain valid in our setting.
First, we observe that combining representation \eqref{eq:part_fct_sum1} together with \eqref{ineq:boundR}, we get
\begin{align}
\exp{[- 2\gamma \eta_{n}({\bf{B}}) n^2]}&\sum_{X\in \mathcal{A}^{({\bf{B}})}_n}\exp{[n^2U^{(\mathcal{B})}_{\boldsymbol{\alpha}, \bf{h}}(X)]} \leq Z^{(\bf{B})}_{n;\boldsymbol{\alpha},\bf{h}}=
\sum_{X\in \mathcal{A}^{(\bf{B})}_{n}}\exp{[n^2U^{(\mathcal{B}_n)}_{\boldsymbol{\alpha}, \bf{h}}(X)]}\\
&=\sum_{X\in \mathcal{A}^{(\bf{B})}_{n}}\exp{[n^2U^{(\mathcal{B})}_{\boldsymbol{\alpha}, \bf{h}}(X)]}\exp{[n^2R_{n;\boldsymbol{\alpha}, \bf{h}}(X)]}\\
&\leq \exp{[ 2\gamma \eta_{n}({\bf{B}}) n^2]}\sum_{X\in \mathcal{A}^{(\bf{B})}_{n}}\exp{[n^2U^{(\mathcal{B})}_{\boldsymbol{\alpha}, \bf{h}}(X)]}. 
\end{align}
Since $\eta_n(\mathbf B)\to 0$ as $n\to\infty$, it suffices to analyze the term
\[
\Psi_n:= \frac{1}{n^2}\ln \sum_{X\in \mathcal{A}^{(\bf{B})}_{n}}\exp{[n^2U^{(\mathcal{B})}_{\boldsymbol{\alpha}, \bf{h}}(X)]},
\]
as the contribution of the remaining terms vanishes in the limit $n\to\infty$.
Fix $\varepsilon>0$. 
Recall that $U^{(\mathcal B)}_{\boldsymbol{\alpha},\mathbf h}$ can be viewed as a function 
$\widetilde{\mathcal{W}}_{\mathcal{B}}^{(k)} \to \mathbb{R}$, and that it is bounded. Therefore, there is a finite set $A$ such that the intervals 
$\{(a, a + \varepsilon) : a \in A\}$ cover the range of $U^{(\mathcal B)}_{\boldsymbol{\alpha},\mathbf h}$. For each $a \in A$, let 
\begin{align}
\widetilde{\mathcal{F}}^{(a,\mathcal{B})} &:= \bigl(U^{(\mathcal B)}_{\boldsymbol{\alpha},\mathbf h}\bigr)^{-1}
([a, a + \varepsilon])\subseteq \widetilde{\mathcal{W}}_{\mathcal{B}}^{(k)}\label{inv_set_U}\\
\widetilde{\mathcal{F}}^{(a,\mathcal{B})}_n &:=\{X^{({\bf{B}})}\in \mathcal{A}^{({\bf{B}})}_n: \tau(g_{\mathcal{B}}^{X})\in \widetilde{\mathcal{F}}^{(a,\mathcal{B})} \}.
\end{align}
By the continuity of $U^{(\mathcal B)}_{\boldsymbol{\alpha},\mathbf h}$, each $\widetilde{\mathcal{F}}^{(a,\mathcal{B})}$ is closed. Now,
\begin{align}
e^{n^2 \Psi_n} &= \sum_{X\in \mathcal{A}^{({\bf{B}})}_n}e^{n^2 U^{(\mathcal B)}_{\boldsymbol{\alpha},\mathbf h}(g^{X}_{\mathcal{B}})}=\sum_{a \in A} \sum_{X^{({\bf{B}})}\in \widetilde{\mathcal{F}}^{(a,\mathcal{B})}_n}e^{n^2 U^{(\mathcal B)}_{\boldsymbol{\alpha},\mathbf h}(g^{X}_{\mathcal{B}})}\leq \sum_{a\in A}e^{n^2(a+\varepsilon)}|\widetilde{\mathcal{F}}^{(a,\mathcal{B})}_n| \notag\\
&\leq |A|\sup_{a\in A}e^{n^2(a+\varepsilon)}|\widetilde{\mathcal{F}}^{(a,\mathcal{B})}_n| .\label{ineq:a}
\end{align}
By Lem.~\ref{ER_LDP} (i.e. the \ER LDP on the space $\widetilde{\mathcal{W}}_{\mathcal{B}}^{(k)}$), 
\begin{equation}\label{ldp_ER_inf}
\limsup_{n \to \infty} \frac{\ln |\widetilde{\mathcal{F}}^{(a,\mathcal{B})}_n|}{n^2} \leq \frac{\ln 2}{2} - \inf_{\tilde{h} \in \widetilde{\mathcal{F}}^{(a,\mathcal{B})}} \mathcal{I}^{(\mathcal{B})}_{\frac{1}{2}}(\tilde{h}) = -\frac{1}{2}\inf_{\tilde{h} \in \widetilde{\mathcal{F}}^{(a,\mathcal{B})}} \mathcal{I}^{(\mathcal{B})}(\tilde{h}).
\end{equation}

Hence, we get
\begin{align}
\limsup_{n\to\infty}\Psi_n&= \limsup_{n\to\infty}\frac{1}{n^2}\ln e^{n^2\Psi_n}\notag\\
&\overset{\eqref{ineq:a}}{\leq}
\limsup_{n \to \infty} \frac{\ln(A)}{n^2}+ \sup_{a\in A}\left[\limsup_{n \to \infty} \frac{1}{n^2}\ln e^{n^2(a+\varepsilon)}+ \limsup_{n \to \infty} \frac{\ln|\widetilde{\mathcal{F}}^{(a,\mathcal{B})}_n|}{n^2} \right] \notag\\
&\overset{\eqref{ldp_ER_inf}}{=}\sup_{a \in A} \left( a + \varepsilon - \frac{1}{2}\inf_{\tilde{h} \in \widetilde{\mathcal{F}}^{(a,\mathcal{B})}} \mathcal{I}^{(\mathcal{B})}(\tilde{h}) \right). \label{ineq:psi:prelim}
\end{align}
Now, from \eqref{inv_set_U}, each $\tilde{h} \in \widetilde{\mathcal{F}}^{(a,\mathcal{B})}$ satisfies $U^{(\mathcal B)}_{\boldsymbol{\alpha},\mathbf h}(\tilde{h}) \geq a$. Consequently,
\begin{align*}
&\sup_{\tilde{h} \in \widetilde{\mathcal{F}}^{(a,\mathcal{B})}} \left(U^{(\mathcal B)}_{\boldsymbol{\alpha},\mathbf h}(\tilde{h}) - \frac{1}{2}\mathcal{I}^{(\mathcal{B})}(\tilde{h})\right) \geq \sup_{\tilde{h} \in \widetilde{\mathcal{F}}^{(a,\mathcal{B})}} \left(a - \frac{1}{2}\mathcal{I}^{(\mathcal{B})}(\tilde{h})\right)\\
&\hspace{3cm}= a + \frac{1}{2}\sup_{\tilde{h} \in \widetilde{\mathcal{F}}^{(a,\mathcal{B})}} [-\mathcal{I}^{(\mathcal{B})}(\tilde{h})] = a - \frac{1}{2}\inf_{\tilde{h} \in \widetilde{\mathcal{F}}^{(a,\mathcal{B})}} \mathcal{I}^{(\mathcal{B})}(\tilde{h}),
\end{align*}
and rearranging terms 
\begin{equation}\label{ineq:ratef}
- \frac{1}{2}\inf_{\tilde{h} \in \widetilde{\mathcal{F}}^{(a,\mathcal{B})}} \mathcal{I}^{(\mathcal{B})}(\tilde{h})\leq -a + \sup_{\tilde{h} \in \widetilde{\mathcal{F}}^{(a,\mathcal{B})}} \left(U^{(\mathcal B)}_{\boldsymbol{\alpha},\mathbf h}(\tilde{h}) - \frac{1}{2}\mathcal{I}^{(\mathcal{B})}(\tilde{h})\right).
\end{equation}
Substituting \eqref{ineq:ratef} into \eqref{ineq:psi:prelim} we get
\begin{align}\label{ineq:feA}
\limsup_{n\to\infty}\Psi_n &\leq \varepsilon + \sup_{a \in A} \sup_{\tilde{h} \in \widetilde{\mathcal{F}}^{(a,\mathcal{B})}} \left(U^{(\mathcal B)}_{\boldsymbol{\alpha},\mathbf h}(\tilde{h}) - \frac{1}{2}\mathcal{I}^{(\mathcal{B})}(\tilde{h})\right)\notag\\
&= \varepsilon + \sup_{\tilde{h} \in \widetilde{\mathcal{W}}_{\mathcal{B}}^{(k)}} \left(U^{(\mathcal B)}_{\boldsymbol{\alpha},\mathbf h}(\tilde{h}) - \frac{1}{2}\mathcal{I}^{(\mathcal{B})}(\tilde{h})\right).
\end{align}
Similarly,
for each $a \in A$, one can define 
$\widetilde{\mathcal{O}}^{(a,\mathcal{B})} := \bigl(U^{(\mathcal B)}_{\boldsymbol{\alpha},\mathbf h}\bigr)^{-1}
((a, a + \varepsilon))$. By retracing the same steps (and using in \eqref{ldp_ER_inf} the opposite inequality for the $\liminf$ of open sets provided by the LDP) one finally concludes that 
\begin{equation}\label{ineq:feB}
\limsup_{n\to\infty}\Psi_n \geq
- \varepsilon + \sup_{\tilde{h} \in \widetilde{\mathcal{W}}_{\mathcal{B}}^{(k)}} \left(U^{(\mathcal B)}_{\boldsymbol{\alpha},\mathbf h}(\tilde{h}) - \frac{1}{2}\mathcal{I}^{(\mathcal{B})}(\tilde{h})\right).
\end{equation}
Since $\varepsilon$ is arbitrary in \eqref{ineq:feA} and \eqref{ineq:feB}, this completes the proof. 
\end{proof}

\bibliographystyle{abbrv}
\bibliography{biblio}
\end{document}